\documentclass[12pt,a4paper,reqno]{amsart}

\usepackage{a4wide}
\usepackage{amssymb}
\usepackage{graphicx}
\usepackage{float}
\usepackage[ansinew]{inputenc}
\theoremstyle{plain}
\newtheorem{theorem}{Theorem}[section]
\newtheorem{maintheorem}{Theorem}
\newtheorem{lemma}[theorem]{Lemma}
\newtheorem{proposition}[theorem]{Proposition}
\newtheorem{corollary}[theorem]{Corollary}
\newtheorem{maincorollary}[maintheorem]{Corollary}

\theoremstyle{definition}
\newtheorem{definition}{Definition}

\theoremstyle{remark}
\newtheorem{remark}[theorem]{Remark}

\newtheorem{claim}{Claim}[section]

\newcommand{\um}{{\bf 1}}

\newcommand{\leb}{\operatorname{Leb}}

\newcommand{\dist}{\operatorname{dist}}

\newcommand{\var}{\operatorname{var}}

\def\R{\ensuremath{\mathbb R}}
\def\N{\ensuremath{\mathbb N}}

\def\I{\ensuremath{{\bf 1}}}

\def\F{\ensuremath{\mathcal F}}
\def\B{\ensuremath{\mathcal BC}}

\def\L{\ensuremath{\mathcal L}}

\def\PP{\ensuremath{\mathbb P}}
\def\f{\ensuremath{F}}
\def\fb{\ensuremath{\overline{F}}}
\def\l0{\ensuremath{\leb_{\hat\gamma_0}}}
\def\ln{\ensuremath{\leb_{\hat\gamma_n}}}

\def\dist{\ensuremath{\mbox{dist}}}

\newcommand{\qand}{\quad\text{and}\quad}

\numberwithin{equation}{section}

\begin{document}

\title[Statistical stability and continuity of SRB entropy]{Statistical stability and continuity of SRB entropy  for systems with Gibbs-Markov structures}
\author[J. F. Alves]{Jos\'{e} F. Alves}
\address{Jos\'{e} F. Alves\\ Centro de Matem\'{a}tica da Universidade do Porto\\ Rua do Campo Alegre 687\\ 4169-007 Porto\\ Portugal}
\email{jfalves@fc.up.pt} \urladdr{http://www.fc.up.pt/cmup/jfalves}

\author[M. Carvalho]{Maria Carvalho}
\address{Maria Carvalho\\ Centro de Matem\'{a}tica da Universidade do Porto\\ Rua do
Campo Alegre 687\\ 4169-007 Porto\\ Portugal}
\email{mpcarval@fc.up.pt}

\author[J. M. Freitas]{Jorge Milhazes Freitas}
\address{Jorge M. Freitas\\ Centro de Matem\'{a}tica da Universidade do Porto\\ Rua do
Campo Alegre 687\\ 4169-007 Porto\\ Portugal}
\email{jmfreita@fc.up.pt}
\urladdr{http://www.fc.up.pt/pessoas/jmfreita}

\date{\today}
\thanks{}

\begin{abstract} We present conditions on families of diffeomorphisms that guarantee statistical stability and SRB entropy continuity. They rely on the existence of horseshoe-like sets with infinitely many branches and variable return times.
As an application we consider  the family of H\'enon maps  within
the set of Benedicks-Carleson parameters.

\end{abstract}

\maketitle

\setcounter{tocdepth}{2}

\tableofcontents 

\section{Introduction}

A \textit{physical measure} for a smooth map $f:M\rightarrow M$ on
a manifold $M$ is a Borel probability measure $\mu$ on $M$ for
which there is a positive Lebesgue measure set of points $x\in M$,
called the \emph{basin} of~$\mu$, such that
\begin{equation}
\label{eq:def-physical-measure}
\lim_{n\rightarrow\infty}\frac{1}{n}\sum_{j=0}^{n-1}\delta_{f^j(x)}\stackrel{n\to \infty}\longrightarrow \mu
\end{equation}
in the weak* topology, where $\delta_z$ stands for the Dirac measure on $z\in M$. Sinai, Ruelle and Bowen showed the existence of physical measures for Axiom A smooth dynamical systems. These were obtained as \emph{equilibrium states} for the logarithm of the Jacobian along the unstable direction. Besides, such probability measures exhibit positive Lyapunov exponents and
conditionals which are absolutely continuous  with respect to Lebesgue measure on local unstable
leaves;
probability measures with the latter  properties are nowadays known as
\textit{Sinai-Ruelle-Bowen measures}\index{SRB measures} (\emph{SRB measures}, for
short).

Statistical properties and their stability have met with wide interest, particularly in the context of dynamical systems which  do not satisfy classical structural stability. This may be checked through the continuous variation of the SRB measures, referred in \cite{AV}  as  \textit{statistical stability}.
Another characterization of stability addresses the continuity of the metric entropy of SRB measures.
Although an old issue, going back to \cite{N} and \cite{Yom} for example, this continuity (topological or metric) is in general a hard problem. Notice that for families of smooth diffeomorphisms verifying the \emph{entropy formula}, see \cite{LY2}, and whose Jacobian along the unstable direction depends continuously on the map, the entropy continuity is an immediate consequence of the statistical stability. This holds for instance in the setting of Axiom~A attractors whose statistical stability was established in \cite{R2} and \cite{Man90}. The regularity of the SRB entropy for Axiom A flows was proved in \cite{C}.
Analiticity of metric entropy for Anosov diffeomorphisms was proved in \cite{Po}.

More recently, statistical stability for families of partially hyperbolic diffeomorphisms with non-uniformly expanding centre-unstable direction was established in \cite{Va}. Due to the continuous variation of the centre-unstable direction in the  partial hyperbolicity context, the entropy continuity follows as in the Axiom A case. Statistical stability for Hénon maps within Benedicks-Carleson parameters have been proved in \cite{ACF}; the entropy continuity for this family is a more delicate issue, since the lack of partial hyperbolicity, mostly due to the presence of ``critical'' points, originates a highly irregular behavior of the unstable direction.
In the endomorphism setting, many advances have been obtained for important families of maps, for instance in \cite{RS,T2,T1,AV,A,F,FT} concerning statistical stability, and in \cite{AOT} for the entropy continuity.  Actually, our main theorem may be regarded as a version for diffeomorphisms of the entropy continuity result in  \cite{AOT}.

In this work we give sufficient conditions on families of smooth diffeomorphisms for the statistical stability and the continuous variation
of the SRB entropies.  The families we study here, though having directions of non-uniform expansion, 
do not allow the approach of the hyperbolic case, since no continuity assumptions on these directions with the map will be assumed.
Instead, we consider diffeomorphisms admitting Gibbs-Markov structures as in \cite{Yo98} that may be thought as ``horseshoes'' with infinitely many branches and variable return times. This is mainly motivated by the important class of Hénon maps presented in the next paragraph. Our assumptions, which have a geometrical and dynamical nature, ensure in particular the existence of SRB measures. Gibbs-Markov structures were used in \cite{Yo98} to derive decay of correlations and the validity of the Central Limit Theorem for the SRB measure. Here we prove that under some additional uniformity requirements on the family we obtain statistical stability and SRB entropy continuity.

The major  application of our main result concerns the Benedicks-Carleson family of H\'enon maps,
\index{H\'enon maps}
\begin{equation} \label{def:Henon-map}f_{a,b}:
\begin{array}[t]{ccc}
\R^2 & \longrightarrow & \R^2\\
(x,y)& \longmapsto & (1-ax^2+y,bx).
\end{array}
\end{equation}
For small $b>0$ values, $f_{a,b}$
is strongly dissipative, and may be seen as an ``unfolded''
version of a quadratic interval map. It is known
that for small $b$ there is a trapping  region whose
topological attractor coincides with the closure of the
unstable manifold $W$ of a fixed point $z^*_{a,b}$ of $f_{a,b}$. In \cite{BC91} it was shown that for each sufficiently small $b>0$ there is
 a positive Lebesgue measure set
of parameters $a\in[1,2]$ for which $f_{a,b}$ has
a dense orbit in $\overline W$ with a positive Lyapunov exponent, which makes this a non-trivial and strange attractor. We denote by $\B$ the set of those parameters $(a,b)$ and call it
the \emph{Benedicks-Carleson family of H\'enon maps}.
As shown in \cite{BY93}, each of these non-hyperbolic attractors
supports a unique SRB measure $\mu_{a,b}$, whose main features were further studied
in
\cite{BY00,BV01,BV06}.  In \cite{BY00}
 a Gibbs-Markov structure was built for each $f_{a,b}$ with $(a,b)\in\B$, which has been used to obtain  statistical behavior  of H\"older observables. These structures have also been used in \cite{ACF} to  deduce the statistical stability of this family.
In this work we add the metric entropy continuity with respect to these measures.

 \subsection{Gibbs-Markov structure}\label{subsec:hypstructures}
Let $f\colon M\to M$ be $C^k$ diffeomorphism ($k\geq2$) defined on a
finite dimensional Riemannian manifold $M$, endowed with a
normalized volume form on the Borel sets that we denote by $\leb$ and call {\em
Lebesgue measure.} Given a submanifold $\gamma\subset M$   we use
$\leb_\gamma$ to denote the measure on $\gamma$ induced by the
restriction of the Riemannian structure to~$\gamma$.

An embedded disk $\gamma\subset M$ is called an {\em unstable
manifold} if $\dist(f^{-n}(x),f^{-n}(y))\to0$  as
$n\to\infty$ for every $x,y\in\gamma$. Similarly, $\gamma$ is called
a {\em stable manifold} if $\dist(f^{n}(x),f^{n}(y))\to0$
 as $n\to\infty$ for every $x,y\in\gamma$.

\begin{definition}
\label{def:cont-families-of-(un)stable-curves} Let
 $D^u$ be the  unit disk in some Euclidean space and
$\text{Emb}^1(D^u,M)$ be the space of $C^1$ embeddings from $D^u$
into $M$. We say that $\Gamma^u=\{\gamma^u\}$ is a \emph{continuous
family of $C^1$ unstable manifolds} if there is a compact set~$K^s$ and $\Phi^u\colon K^s\times
D^u\to M$ such that
\begin{itemize}
\item[i)] $\gamma^u=\Phi^u(\{x\}\times D^u)$ is an unstable
manifold;
\item[ii)] $\Phi^u$ maps $K^s\times D^u$ homeomorphically onto its
image; \item[iii)] $x\mapsto \Phi^u\vert(\{x\}\times D^u)$ defines a
continuous map from $K^s$ into $\text{Emb}^1(D^u,M)$.
\end{itemize}
Continuous families of $C^1$ stable manifolds are defined similarly.
\end{definition}


\begin{definition}
\label{def:horseshoe}
 We say that $\Lambda\subset M$ has a \emph{hyperbolic product
structure} if there exist a continuous family of unstable manifolds
$\Gamma^u=\{\gamma^u\}$ and a continuous family of stable manifolds
$\Gamma^s=\{\gamma^s\}$ such that
\begin{itemize}
    \item[i)] $\Lambda=(\cup \gamma^u)\cap(\cup\gamma^s)$;
    \item[ii)] $\dim \gamma^u+\dim \gamma^s=\dim M$;
    \item[iii)] each  $\gamma^s$ meets each $\gamma^u$  in exactly
    one point;
    \item[iv)] stable and unstable manifolds meet with angles
    larger than some $\theta>0$.
\end{itemize}
\end{definition}

Let $\Lambda\subset M$ have a hyperbolic product structure, whose
defining families are $\Gamma^s$ and $\Gamma^u$. A subset
$\Upsilon_0\subset \Lambda$ is called an {\em $s$-subset} if
$\Upsilon_0$ also has a hyperbolic product structure and its
defining families $\Gamma_0^s$ and $\Gamma_0^u$ can be chosen with
$\Gamma_0^s\subset\Gamma^s$ and $\Gamma_0^u=\Gamma^u$; {\em
$u$-subsets} are defined analogously. Given $x\in\Lambda$, let
$\gamma^{*}(x)$ denote the element of $\Gamma^{*}$ containing $x$,
for $*=s,u$. For each $n\ge 1$, let $(f^n)^u$ denote the restriction
of the map $f^n$ to $\gamma^u$-disks, and let $\det D(f^n)^u$ be the
Jacobian of $D(f^n)^u$. In the  sequel $C>0$ and $0<\beta<1$ are constants, and we require the following properties from the hyperbolic product
structure~$\Lambda$:

\renewcommand{\theenumi}{\textbf{P$_{\arabic{enumi}}$}}

\begin{enumerate}
\setcounter{enumi}{-1}
    \item
    \label{P-positive-measure}\emph{Positive measure}: for every $\gamma\in\Gamma^u$
    we have $\leb_\gamma(\Lambda\cap\gamma)>0$.
    \item
\label{P-Markov}  \emph{Markovian}: there are pairwise disjoint
$s$-subsets $\Upsilon_1,\Upsilon_2,\dots\subset\Lambda$ such
    that
    \begin{enumerate}
 \item $\leb_{\gamma}\big((\Lambda\setminus\cup\Upsilon_i)\cap\gamma\big)=0$ on each $\gamma\in\Gamma^u$;
 \item for each $i\in\N$ there is $\tau_i\in\N$ such that $f^{\tau_i}(\Upsilon_i)$ is a $u$-subset,
         and for all $x\in \Upsilon_i$
         $$
         f^{\tau_i}(\gamma^s(x))\subset \gamma^s(f^{\tau_i}(x))\qand
         f^{\tau_i}(\gamma^u(x))\supset \gamma^u(f^{\tau_i}(x));
         $$
 \item for each $n\in\N$ there are finitely many $i$'s with
 $\tau_i=n$.

    \end{enumerate}
\end{enumerate}

 \begin{enumerate}
 \setcounter{enumi}{1}
\item
\label{P-contraction-s-leaves} \emph{Contraction on  stable leaves}:
 for each $\gamma^s\in\Gamma^s$ and each $ y\in\gamma^s(x)$
 $$\dist(f^n(y),f^n(x))\le C\beta^n,\quad  \forall n\ge
 1. $$
 \end{enumerate}
For the last two properties we introduce the \emph{return time} $R\colon
\Lambda\to\N$ and the \emph{induced map}  $F=f^R\colon
\Lambda\to\Lambda$, which are defined for each
$i\in\N$ as $$
 R\vert_{\Upsilon_i}=\tau_i\qand
f^R\vert_{\Upsilon_i}=f^{\tau_i}\vert_{\Upsilon_i}, $$
and, for each   $x,y\in\Lambda$,
the {\em separation time} $s(x,y)$ is given by
$$
s(x,y)=\min\left\{n\ge0 \colon \textrm{$(f^R)^n(x)$ and $(f^R)^n(y)$
lie in distinct }\Upsilon_{i}'s
\right\}. $$

\begin{enumerate}
\setcounter{enumi}{2}
\item
\label{P-regularity-s-foliation} {\em Regularity of the stable
foliation}:
\begin{enumerate}
\item for $y\in\gamma^s(x)$ and $n\ge0$
  $$
    \log\prod_{i=n}^\infty\frac{\det Df^u(f^i(x))}{\det
        Df^u(f^i(y))}\le C\beta^n;$$
 \item given
$\gamma,\gamma'\in\Gamma^u$, we define
$\Theta\colon\gamma'\cap\Lambda\to\gamma\cap\Lambda$  by
$\Theta(x)=\gamma^s(x)\cap \gamma$. Then $\Theta$ is absolutely continuous and
        $$\displaystyle \frac{d(\Theta_*\leb_{\gamma'})}{d\leb_{\gamma}}(x)=
        \prod_{i=0}^\infty\frac{\det Df^u(f^i(x))}{\det
        Df^u(f^i(\Theta^{-1}(x)))};$$
  \item letting $v(x)$ denote the density in item (b),  we have
 $$\log \frac{v(x)}{v(y)}\le C\beta^{s(x,y)},\quad\text{for
 $x,y\in\gamma'\cap\Lambda$}.$$
\end{enumerate}
\end{enumerate}

\begin{enumerate}
\setcounter{enumi}{3}
 \item
 \label{P-bounded-distortion}
 {\em Bounded distortion}:
 for $\gamma\in\Gamma^u$ and
    $x,y\in\Lambda\cap\gamma$
    $$
    \log\frac{\det D(f^R)^u(x)}{\det D(f^R)^u(y)}\le
    C\beta^{s(f^R(x),f^R(y))}.
    $$
 \end{enumerate}

\begin{remark}\label{rem:LSY-slight-changes} We do not assume uniform backward
contraction along unstable leaves as (P4)(a) in \cite{Yo98}.
Properties \eqref{P-regularity-s-foliation}(c) and
\eqref{P-bounded-distortion} are new if comparing our setup to that in
\cite{Yo98}. However,
these are consequence of (P4) and  (P5) of
 \cite{Yo98} as done in \cite[Lemma~1]{Yo98}.\end{remark}

In spite of the uniform contraction on stable leaves demanded in \eqref{P-contraction-s-leaves}, this
is not too restrictive in systems having regions where the
contraction fails to be uniform, since we are allowed to remove
stable leaves, provided a subset with positive
measure of leaves remains in the end. This has been carried
out for H\'{e}non maps in \cite{BY00}.

\subsection{Uniform families}
\label{subsec:unif-fam-maps}
Let $\F$ be a a family of $C^k$ maps ($k\geq2$) from the finite
dimensional Riemannian manifold $M$ into itself, and endow $\F$ with
the $C^k$ topology. Assume that each map $f\in\F$ admits a Gibbs-Markov structure $\Lambda_f$ as described in
Section~\ref{subsec:hypstructures}. Let $\Gamma^u_f=\{\gamma^u_f\}$
and $\Gamma^s_f=\{\gamma^s_f\}$ be its defining families of unstable
and stable curves. Denote by $R_f:\Lambda_f\to\N$ the corresponding
return time function. 

Given
    $f_0\in\F$, take a sequence $f_n\in\F$ such that $f_n\to f_0$ in
    the $C^1$ topology as $n\to\infty$. For the sake of notational simplicity, for each $n\ge 0$ we will indicate the dependence of the previous objects  on $f_n$  just by means of the index or supra-index $n$.
    If $\gamma^u_n\in \Gamma^u_n$ is sufficiently close to $\gamma^u_0\in \Gamma^u_0$ in the $C^k$ topology,
    we may define a projection by
    sliding through the stable manifolds of $\Lambda_0$
    $$
    \begin{array}{cccc}
      H_n: & \gamma^u_{ n}\cap\Gamma^s_{0} & \longrightarrow & \gamma^u_{ 0} \\
      & z & \longmapsto & \gamma^s_0(z)\cap\gamma^u_{ 0}
    \end{array}
    $$
and set

 \begin{equation}\label{eq.omegan}
 \Omega_{0}=\gamma^u_{0}\cap\Lambda_{0},\quad\Omega_n^0=H_n^{-1}(\Omega_0)
 ,\quad\Omega_{n}=\gamma^u_{n}\cap\Lambda_{n},\quad\Omega^n_0=H_n(\Omega_n\cap\Omega_n^0).
 \end{equation}
Given $k\in\N$ and positive integers $i_1,\ldots,i_k$,
we denote by $\Upsilon_{i_1,\ldots,i_k}^n$ the $s$-sublattice that
satisfies $F^j_n(\Upsilon_{i_1,\ldots,i_k}^n)\subset\Upsilon_{i_j}^n$
for every $1\le j<k$ and
$F^k_n(\Upsilon_{i_1,\ldots,i_k}^n)=\Upsilon_{i_k}^n$.

\begin{definition}\label{def:uniform-family}
$\F$  is called a \emph{uniform family} if the conditions \eqref{u-uniform-constants}--\eqref{u-uniform-tail}
below hold:

\renewcommand{\theenumi}{\textbf{U$_{\arabic{enumi}}$}}

\begin{enumerate}
\setcounter{enumi}{-1}
    \item
\label{u-uniform-constants}
    \emph{Absolute constants:} the constants $C$ and
$\beta$   in
\eqref{P-contraction-s-leaves},\eqref{P-regularity-s-foliation} and
\eqref{P-bounded-distortion}
can be chosen the same for all
$f\in\F$.
\end{enumerate}

\begin{enumerate}
    \item
\label{u-prox-unstable-direction}
    \emph{Proximity of unstable leaves:} there are  unstable leaves
    $\hat\gamma_{0}\in\Gamma^u_{ 0}$ and
    $\hat\gamma_{ n}\in\Gamma_{ n}$ such that
    $\hat\gamma_{ n}\to\hat\gamma_{ 0}$ in the $C^1$
    topology as $n\to\infty$.
\end{enumerate}

\begin{enumerate}
\setcounter{enumi}{1}
\item
\label{u-matching-cantor-sets}
    \emph{Matching of structures:} defining the objects of \eqref{eq.omegan} with $\hat\gamma_n$ replacing $\gamma_n^u$, we have
    $$\ln\left(\Omega_{n}\triangle
    \Omega_n^0\right)\to0,\quad \text{as $n\to\infty$}.$$
    \item
\label{u-proximity-stable-direction} \emph{Proximity of stable
directions:} for every
    $z\in \Omega_0^n\cap\Omega_{0}$ we have $\gamma^s_{n}(z)\to
    \gamma^s_{0}(z)$ in the $C^1$ topology as $n\to\infty$.
\end{enumerate}
\begin{enumerate}
\setcounter{enumi}{3}
    \item
\label{u-matching-s-sublattices}
    \emph{Matching of s-sublattices:} given $N,k\in\N$ and
    $\Upsilon_{i_1,\ldots,i_k}^{ 0}$ with
    $R_{ 0}\big(\Upsilon_{i_j}^{ 0}\big)\leq N$ for $1\le j\le
    k$,
    there is
    $\Upsilon_{\ell_1,\ldots,\ell_k}^{ n}$ such that
    $R_{ n}\big(\Upsilon_{\ell_j}^{ n}
    \big)=R_{ 0}\big(\Upsilon_{i_j}^{ 0}\big)$ for
    $1\le j\le k$
    and
    $$\leb_{\hat\gamma_{ 0}}\left(H_n\big(
    \Upsilon_{\ell_1,\ldots,\ell_k}^{ n}\cap
    \hat\gamma_{ n}\big)\triangle\big(\Upsilon_{i_1,\ldots,i_k}^{ 0}
    \cap
    \hat\gamma_{ 0}\big) \right)\to 0,\quad\text{as $n\to\infty$.}$$
\end{enumerate}

\begin{enumerate}
\setcounter{enumi}{4}
    \item
\label{u-uniform-tail}
    \emph{Uniform tail}:  given $\varepsilon>0$, there are $N=N(\varepsilon)$
and $J=J(\varepsilon,N)$ such that
 $$
 \sum_{j=N}^\infty j\ln \{R_n=j\}<\varepsilon, \quad\forall n>J.
 $$
\end{enumerate}

\end{definition}
This last property ensures in particular that
$\int_{\hat\gamma_n}R_n\,d\leb_{\hat\gamma_n}<\infty$ for large $n$, which by
\cite[Theorem~1]{Yo98} implies
the existence of an SRB measure  for each $f_n$.

\renewcommand{\theenumi}{\arabic{enumi}}

  \begin{remark}
    \label{rem:Leb-Hm*ln-l0}
    Using that stable and unstable manifolds of $f_0$ meet with angles
    uniformly bounded away from zero at points in $\Lambda_0$, and the proximities given by \eqref{u-prox-unstable-direction} and \eqref{u-proximity-stable-direction}, it follows that there is some $\theta>0$ such that, for $n$ large enough, the stable manifolds through points in $\Omega_n^0$ meet $\hat\gamma_n$ with an angle bigger than $\theta$. Together with  \eqref{P-regularity-s-foliation} and \eqref{u-prox-unstable-direction}, this implies that:
    \begin{itemize}
      \item[\emph{i)}] $(H_n)_*\ln\ll\l0$ with uniformly bounded density;
      \item[\emph{ii)}] $\frac{d(H_n)_*\ln}{\l0}\to 1$
    on $L^1(\l0)$, as $n\to\infty$.
    \end{itemize}

\end{remark}

\subsection{Statement of results} %

Consider a family $\F$  such that each $f\in\F$ admits a unique
SRB measure $\mu_f$.
Letting $\PP(M)$ denote the space of probability measures on $M$ endowed
with the weak* topology, we say that $\F$ is
\emph{statistically stable} if the map
 \[\begin{array}{ccc}
                 \F & \longrightarrow &
                 $\PP(M)$\\
                    f  & \mathbf{\longmapsto} &
                    \mu_f,
                    \end{array}
\]
is continuous. In the sequel $h_{\mu_f}$ denotes the metric entropy
of $f$ with respect to the measure~$\mu_f$.
\begin{maintheorem}
Let $\F$ be a uniform family such that
each $f\in\F$ admits a unique SRB
measure. Then
\begin{enumerate}
  \item $\F$ is statistically stable;
\item $\F\ni f\mapsto h_{\mu_f}$ is
continuous.
\end{enumerate}
\end{maintheorem}

\begin{maincorollary}\label{thm:statstab-entcont-henon}
The family $\B$ is statistically stable and the map  $\B \ni (a,b)  \mapsto h_{\mu_{a,b}}$  is continuous.
\end{maincorollary}

This corollary follows immediately after building Gibbs-Markov structures satisfying \eqref{P-positive-measure}--\eqref{P-bounded-distortion}, as was done in \cite{BY00}, and verifying the uniformity conditions \eqref{u-uniform-constants}--\eqref{u-uniform-tail}, as in \cite{ACF}. For the sake of clearness, the following list specifies exactly where each property is obtained.

\begin{center}
\begin{tabular}{|l|l|}
\hline
\eqref{P-positive-measure}&\cite[Proposition~A(3)]{BY00}\\
\hline
\eqref{P-Markov}&\cite[Proposition~A(1),(2)]{BY00}\\
\hline
\eqref{P-contraction-s-leaves}&\cite[Proposition~A(2)]{BY00}\\
\hline
\eqref{P-regularity-s-foliation}(a)&\cite[Sublemma~8]{BY00}\\
\hline
\eqref{P-regularity-s-foliation}(b)&\cite[Sublemma~10]{BY00}\\
\hline
\eqref{P-regularity-s-foliation}(c)& \cite[Sublemma~11]{BY00}\\
\hline
\eqref{P-bounded-distortion}& \cite[Sublemma~9]{BY00}\\
\hline
\eqref{u-uniform-constants}&\cite[Sections~6,7,8]{ACF}\\
\hline
\eqref{u-prox-unstable-direction}&Hyperbolicity of the fixed point $z^*$\\
\hline
\eqref{u-matching-cantor-sets}&
\cite[Section~6 in particular Corollary~6.4]{ACF}\\
\hline
\eqref{u-proximity-stable-direction}&\cite[Section~7 in particular Proposition~7.3]{ACF}\\
\hline
\eqref{u-matching-s-sublattices}&\cite[Section~8 in particular Proposition~8.9]{ACF}\\
\hline
\eqref{u-uniform-tail}&\cite[Proposition~A(4)]{BY00} \\
\hline
\end{tabular}
\end{center}
\medskip

 Concerning \eqref{u-uniform-constants} and \eqref{u-uniform-tail}, observe that the constants depend exclusively on the maximum value for $b>0$  and the minimum for  $a<2$  in the choice of Benedicks-Carleson parameters.

\section{Quotient dynamics and lifting back}\label{sec:construction-SRB}

In this section we shall analyze some dynamical features of a diffeomorphism $f$ admitting $\Lambda$ with a Gibbs-Markov structure that verifies properties \eqref{P-positive-measure}-\eqref{P-bounded-distortion}.
Consider a quotient space
$\bar{\Lambda}$ obtained by collapsing the stable curves of
$\Lambda$; i.e. $\bar{\Lambda}=\Lambda/\sim$, where $z\sim z'$ if
and only if $z'\in\gamma^s(z)$. Since by \eqref{P-Markov}(b) the
induced map $F=f^R:\Lambda\to\Lambda$ takes $\gamma^s$ leaves to
$\gamma^s$ leaves, then the \emph{quotient induced map}
$\overline{F}:\bar\Lambda\to\bar\Lambda$ is well defined and if
$\bar\Upsilon_i$ is the quotient of $\Upsilon_i$, then $\overline{F}$ takes
the sets $\bar\Upsilon_i$ homeomorphically onto $\bar\Lambda$.  Given an unstable leaf $\gamma$, the set
$\gamma\cap \Lambda$ suits as a model for $\bar{\Lambda}$ through the canonical projection $\bar\pi:\Lambda\to\bar\Lambda$. We will see in
Section~\ref{subsec:quotient-natural-measure} that we may define a natural
reference measure $\bar m$ on $\bar\Lambda$. Besides, $\overline{F}$ is
an expanding Markov map (see Lemma~\ref{lem:jac-def}), thus having an absolutely continuous (w.r.t $\bar m$),
$\overline{F}$-invariant probability measure~$\bar\mu$. Moreover,
if $\tilde\mu$ denotes the $\f$-invariant measure supported on
$\Lambda$ then $\bar\mu=\bar\pi_*(\tilde\mu)$.

  To build an SRB measure $\mu$ out of $\tilde\mu$ is just a matter of
saturating the measure $\tilde\mu$.
The existence of the measures $\bar\mu$, $\tilde\mu$ and the fact that
$\bar\mu=\bar\pi_*(\tilde\mu)$ follows from standard methods, which can be found for instance in \cite{Yo98}. For the sake of
completeness we will present the construction of the SRB measure, also having in mind how
some properties can be carried up through the lifting. We will
accomplish this by adapting some ideas used in the
construction of Gibbs states; see \cite{Bo75}.

\subsection{The natural measure}
\label{subsec:quotient-natural-measure} The purpose of this
subsection is to introduce a natural probability measure $\bar m$ on
$\bar\Lambda$  and establish some properties of the Jacobian of $\overline{F}$
with respect to $\bar m$. Moreover,
we show the existence of an $\overline{F}$-invariant density
$\bar\rho$ with respect to the measure $\bar m$.

Fix an arbitrary $\hat\gamma\in\Gamma^u$. The
restriction of $\bar\pi$ to $\hat\gamma\cap\Lambda$ gives a
homeomorphism that we denote by
$\hat\pi:\hat\gamma\cap\Lambda\to \bar\Lambda$.
Given $\gamma\in\Gamma^u$ and $x\in\gamma\cap\Lambda$ let~$\hat x$
be the point in $\gamma^s(x)\cap \hat\gamma$. Defining for
$x\in\gamma\cap\Lambda$
\begin{equation}\label{eq:def-u-hat}\hat
u(x)=\prod_{i=0}^\infty\frac{\det Df^u(f^i(x))}{\det
        Df^u(f^i(\hat x))}
\end{equation}
we have that $\hat u$ satisfies the bounded distortion property
\eqref{P-regularity-s-foliation}(c). For each $\gamma\in\Gamma^u$
let $m_\gamma$ be the measure in $\gamma$ such that
 $$\frac{dm_\gamma}{d\leb_\gamma}=\hat u\,\um_{\gamma\cap\Lambda},$$
where $\um_{\gamma\cap\Lambda}$  is the characteristic function of
the set ${\gamma\cap\Lambda}$. These measures have been defined in
such a way that if $\gamma,\gamma'\in\Gamma^u$ and $\Theta$ is
obtained by sliding along stable leaves from $\gamma\cap\Lambda$ to
$\gamma'\cap\Lambda$, then
 \begin{equation}\label{eq:bar-m-def} \Theta_*m_{\gamma}=m_{\gamma'}.
\end{equation}
To verify this let us show that the densities of these two measures
with respect to $\leb_{\gamma}$ coincide. Take
$x\in\gamma\cap\Lambda$ and  $x'\in\gamma'\cap\Lambda$ such that
$\Theta(x)=x'$.  By
 \eqref{P-regularity-s-foliation}(b) one has
   $$\frac{d\Theta_*\leb_{\gamma}}{d\leb_{\gamma'}}(x')=\frac{\hat u(x')}{\hat u(x)},$$
   which implies that
  $$\frac{d\Theta_*m_{\gamma}}{d\leb_{\gamma'}}(x')={\hat u(x)}
  \frac{d\Theta_*\leb_\gamma}{d\leb_{\gamma'}}(x')={\hat u(x')}=
 \frac{dm_{\gamma'}}{d\leb_{\gamma'}}(x').$$
 Conditions \eqref{P-positive-measure} and \eqref{eq:bar-m-def} allow
 us to define the reference probability measure $\bar m$ whose
 representative in each unstable leaf $\gamma\in\Gamma^u$ is
 exactly
 $\frac{1}{\leb_{\hat\gamma}(\Lambda)}m_\gamma$.

 Let $T:(X_1,m_1)\to(X_2,m_2)$ be a  measurable bijection
 between two probability measure spaces. $T$ is called
 \emph{nonsingular}  if it maps sets of zero $m_1$ measure to sets of
 zero $m_2$ measure. For a nonsingular transformation $T$ we define the
 Jacobian of $T$ with respect to $m_1$ and $m_2$, denoted by
 $J_{m_1,m_2}(T)$, as the Radon-Nikodym derivative
 $\frac{dT_*^{-1}(m_2)}{dm_1}$. By assertion~(1) of the following lemma
 it makes sense to consider the Jacobian of the quotient map $\overline{F}:(\overline\Lambda,\overline m)\to (\overline\Lambda,\overline m)$ that we simply denote $J\overline{F}$.

\begin{lemma}\label{lem:jac-def}  Assuming that
$F(\gamma\cap\Upsilon_i)\subset\gamma'$ for
$\gamma,\gamma'\in\Gamma^u$, let $JF(x)$ denote the Jacobian of
$F$ with respect to the measures $m_\gamma$ and $m_{\gamma'}$.
Then
\begin{enumerate}
    \item $JF(x)=JF(y)$ for every $y\in\gamma^s(x)$;
    \item there is $C_0>0$ such that for every
    $x,y\in\gamma\cap\Upsilon_i$
    $$
    \left|\frac{JF(x)}{JF(y)}-1\right|\le
    C_0\beta^{s(F(x),F(y))};
    $$
    \item for every $k\in\N$ and any $k$ positive integers
    $i_1,\ldots i_k$, there is $C_1>0$ such that for every $x,y\in
    \Upsilon_{i_1,\ldots,i_k}\cap\gamma$
    $$
    \left|\frac{JF^k(x)}{JF^k(y)}\right|\leq
    C_1.
    $$
\end{enumerate}
\end{lemma}
\begin{proof} (1) For $\leb_\gamma$ almost every $x\in\gamma\cap\Lambda$ we
 have
  \begin{equation}\label{eq:jacobians-rel}
  JF(x)=\left|\det D F^u(x)\right|\cdot \frac{\hat u(F(x))}{\hat u(x)}.
  \end{equation}
 Denoting $\varphi(x)=\log|\det Df^u(x)|$ we may write
 $$
 \begin{array}{lllll}
  \log JF(x) &=& \displaystyle\sum_{i=0}^{R-1}\varphi(f^i(x))&+&
  \displaystyle\sum_{i=0}^{\infty}\left(\varphi(f^i(F(x)))- \varphi(f^i(\widehat{F(x)})\right)\\
     & &   & - & \displaystyle\sum_{i=0}^{\infty}\left(\varphi(f^i(x))- \varphi(f^i(\hat
     x)\right)\\
     &=& \displaystyle\sum_{i=0}^{R-1}\varphi(f^i(\hat x))&+&
  \displaystyle\sum_{i=0}^{\infty}\left(\varphi(f^i(F(\hat x)))- \varphi(f^i(\widehat{F(x)})\right).
 \end{array}
  $$
Thus we have shown that $JF(x)$ can be expressed just in terms of
$\hat x$ and $\widehat{F(x)}$, which is enough for proving the
first part of the lemma.
\smallskip

 (2) It follows from \eqref{eq:jacobians-rel} that
  $$
  \log\frac{JF(x)}{JF(y)}=\log\frac{\det DF^u(x)}{\det
  DF^u(y)}+
  \log\frac{\hat u(F(x))}{\hat u(F(y))}+\log\frac{\hat u(y)}{\hat u(x)}.
  $$
Observing that $s(x,y)> s(F(x),F(y))$ the conclusion follows
from \eqref{P-regularity-s-foliation}(c) and
\eqref{P-bounded-distortion}.

(3) Again, from \eqref{eq:jacobians-rel}, we obtain
  $$
  \log\frac{JF^k(x)}{JF^k(y)}=\log\frac{\det D\left(F^k
  \right)^u(x)}{\det
  D\left(F^k
  \right)^u(y)}+
  \log\frac{\hat u(F^k(x))}{\hat u(F^k(y))}+
  \log\frac{\hat u(y)}{\hat u(x)}.
  $$
By \eqref{P-bounded-distortion} we have
$$\log\frac{\det D\left(F^k
  \right)^u(x)}{\det
  D\left(F^k
  \right)^u(y)}\leq\sum_{l=1}^k C\beta^{s(F^l(x),F^l(y))}\leq
C\sum_{l=0}^\infty \beta^l<\infty.$$
 The remaining terms are easily controlled once again due to \eqref{P-regularity-s-foliation}(c).
 \end{proof}

\begin{lemma}\label{lem:small-adjust}
The map $\overline{F}:\bar\Lambda\to\bar\Lambda$ has an invariant
probability measure $\bar\mu$ with $d\bar\mu=\bar\rho d\bar
m$, where $K^{-1}\leq\bar\rho\leq K$, for some $K=K(C_1,\beta)>0$.
\end{lemma}

\begin{proof}
We construct $\bar{\rho}$ as the density with respect to $\bar{m}$
of an accumulation point of
$\bar{\mu}^{(n)}=1/n\sum_{i=0}^{n-1}\fb_*^i(\bar{m})$. Let
$\bar{\rho}^{(n)}$ denote the density of $\bar{\mu}^{(n)}$ and
$\bar{\rho}^i$ the density of $\fb^i_*(\bar{m})$. Also, let
$\bar{\rho}^i=\sum_{j}\bar{\rho}_j^i$, where $\bar{\rho}_j^i$ is the
density of $\fb_*^i(\bar{m}|\sigma_j^i)$ and the $\sigma_j^i$'s
range over all components of $\bar{\Lambda}$ such that
$\fb^i(\sigma_j^i)=\bar{\Lambda}$.

Consider the normalized  density
$\tilde{\rho}_j^i=\bar{\rho}_j^i/\bar{m}(\sigma_j^i)$. We have for
$\bar{x}'\in\sigma_j^i$ such that $\bar{x}=\fb^i(\bar{x}')$ and for
some $\bar{y}'\in\sigma_j^i$
$$
\tilde{\rho}_j^i(\bar{x})=\frac{J\fb^i(\bar{y}')} {J\fb^i(\bar{x}')}
(\bar{m}(\bar{\Lambda}))^{-1}=\prod_{k=1}^{i}\frac{J\fb(\fb^{k-1}
(\bar{y}'))}{J\fb(\fb^{k-1}(\bar{x}'))}.
$$
By Lemma~\ref{lem:jac-def}(2) we have for every $k=1,\ldots,i$
\[
\frac{J\fb(\fb^{k-1} (\bar{y}'))}{J\fb(\fb^{k-1}(\bar{x}'))}\leq
\exp\left\{C_1\beta^{s\big(\fb^k(\bar y'),\fb^k(\bar
x')\big)}\right\}\leq \exp\left\{C_1\beta^{(i-k)+s(\bar x,\bar
y)}\right\},
\]
from where we conclude that
\[
\tilde{\rho}_j^i(\bar{x})\leq\exp\left\{C_1\beta^{s(\bar x,\bar
y)}\sum_{j\geq0}
\beta^j\right\}\leq\exp\left\{C_1/(1-\beta)\right\}=K.
\]
Observe that we also get
\[
\frac1{\tilde{\rho}_j^i(\bar{x})}=\frac{J\fb^i(\bar{x}')}
{J\fb^i(\bar{y}')} (\bar{m}(\bar{\Lambda}))\leq K,
\]
which yields $\tilde{\rho}_j^i(\bar{x})\geq K^{-1}$.
  Now, since \( \bar{\rho}^i=\sum_j
\bar{m}(\sigma_j^i)\tilde{\rho}_j^i\), we have
$K^{-1}\leq\bar{\rho}^i\leq K$ which implies that
$K^{-1}\leq\bar{\rho}^{(n)}\leq K$, from where we obtain that
$K^{-1}\leq\bar{\rho}\leq K$.
\end{proof}

\subsection{Lifting to the Gibbs-Markov
structure} \label{subsec:stat-stab-raising-measures} We now adapt standard techniques for lifting the $\fb$- invariant measure on the quotient
space  to an $\f$- invariant measure on the initial
Gibbs-Markov structure.

Given an $\fb$-invariant probability measure $\bar{\mu}$, we define
a probability measure $\tilde{\mu}$ on $\Lambda$ as follows. For
each bounded $\phi:\Lambda\rightarrow\R$  consider its
discretizations
$\phi^\bullet:\hat \gamma\cap\Lambda\rightarrow \R$ and
$\phi^*:\bar{\Lambda}\rightarrow \R$ defined by
\begin{equation}
\label{eq:def-discretization}
\phi^\bullet(x)=\inf\{\phi(z):z\in\gamma^s(x)\},\quad \mbox{and}\quad
\phi^*=\phi^\bullet\circ\hat\pi^{-1}.
\end{equation}
If $\phi$ is continuous, as its domain is compact, we may define
$$\var\phi(k)=\sup\left\{|\phi(z)-\phi(\zeta)|:|z-\zeta|\leq
C\beta^k\right\},$$ in which case $\var\phi(k)\rightarrow 0$
as $k\rightarrow\infty$.

\begin{lemma}
\label{lem:bowen-cauchy-sequence} Given any continuous
$\phi:\Lambda\rightarrow\R$,  for all $k,l\in\N$  we have
\[\left|\int(\phi\circ
F^k)^*d\bar{\mu}-\int(\phi\circ
F^{k+l})^*d\bar{\mu}\right|\leq \var\phi(k).
\]
\end{lemma}

\begin{proof}
Since $\bar{\mu}$ is $\fb$-invariant
\begin{equation*}
\begin{split}
\left|\int(\phi\circ F^k)^*d\bar{\mu}-\int(\phi\circ
F^{k+l})^*d\bar{\mu}\right|&= \left|\int(\phi\circ
F^k)^*\circ\fb^ld\bar{\mu}-\int(\phi\circ
F^{k+l})^*d\bar{\mu}\right|\\
&\leq \int\left|(\phi\circ F^k)^*\circ\fb^l-(\phi\circ
F^{k+l})^*\right|d\bar{\mu}.
\end{split}
\end{equation*}
By definition of the discretization we have
\[
(\phi\circ F^k)^*\circ\fb^l(x)=\inf\left\{\phi(z) :z\in
F^k\left(\gamma^s(\fb^l(x))\right)\right\}
\]
and
\[
(\phi\circ F^{k+l})^*(x)=\inf\left\{\phi(\zeta) :\zeta\in
F^{k+l}\left(\gamma^s\left(x\right)\right)\right\}.
\]
Observe that $F^{k+l}\left(\gamma^s\left(x\right)\right)\subset
F^k\left(\gamma^s(\fb^l(x))\right)$ and by
\eqref{P-contraction-s-leaves}
$$\mbox{diam}\, F^k\left(\gamma^s(\fb^l(x)
)\right)~\leq C\beta^{k}.$$  Thus, $\left|(\phi\circ
F^k)^*\circ\fb^l-(\phi\circ
F^{k+l})^*\right|\leq\var\phi(k)$.
\end{proof}
By the Cauchy criterion the sequence $\left(\int(\phi\circ
F^k)^*d\bar{\mu}\right)_{k\in\N}$ converges. Hence, Riesz
Representation Theorem yields a probability measure $\tilde{\mu}$ on
$\Lambda$
\begin{equation}
\label{eq:def-bowen-measure} \int\phi
d\tilde{\mu}:=\lim_{k\rightarrow\infty}\int(\phi\circ
F^k)^*d\bar{\mu}
\end{equation}
for every continuous function $\phi:\Lambda\rightarrow\R$.
\begin{proposition}
\label{prop:properties-bowen-measure} The probability measure
$\tilde{\mu}$ is $\f$-invariant and has absolutely continuous
conditional measures on $\gamma^u$  leaves. Moreover, given any
continuous $\phi:\Lambda\rightarrow\R$ we have
\begin{enumerate}

\item \label{properties-bowen-measure-2-varfik} $\left|\int\phi d\tilde{\mu}-\int(\phi\circ
F^k)^*d\bar{\mu}\right|\leq\var\phi(k)$;

\item \label{properties-bowen-measure-3-discrete}
if $\phi$ is constant in each $\gamma^s$, then $\int\phi
d\tilde{\mu}=\int\bar{\phi} d\bar{\mu}$, where
$\bar{\phi}:\bar{\Lambda}\rightarrow\R$ is defined by
$\bar{\phi}(x)=\phi(z)$, where $z\in\bar\pi^{-1}(x)$;

\item \label{properties-bowen-measure-4-discrete-times-cont}
if $\phi$ is constant in each $\gamma^s$ and
$\psi:\Lambda\rightarrow\R$ is continuous, then
$$\left|\int\psi.\phi d\tilde{\mu}-\int(\psi\circ F^k)^*(\phi\circ
F^k)^*d\bar{\mu}\right|\leq\|\phi\|_1\var\psi(k).$$

\end{enumerate}
\end{proposition}

\begin{proof}
Regarding the $F$-invariance property, note that for any
continuous $\phi:\Lambda\rightarrow\R$,
\[
\int\phi\circ
Fd\tilde{\mu}=\lim_{k\rightarrow\infty}\int\left(\phi\circ
F^{k+1}\right)^*d\bar{\mu}=\int\phi d\tilde{\mu},
\]
by Lemma \ref{lem:bowen-cauchy-sequence}. Assertion
\eqref{properties-bowen-measure-2-varfik} is an immediate
consequence of Lemma \ref{lem:bowen-cauchy-sequence}. Property
\eqref{properties-bowen-measure-3-discrete} follows from
\[
\int\phi d\tilde{\mu}= \lim_{k\rightarrow\infty}\int\left(\phi\circ
F^k\right)^*d\bar{\mu}=\lim_{k\rightarrow\infty}\int
\bar{\phi}\circ\bar{F}^k d\bar{\mu}=\int \bar{\phi}d\bar{\mu},
\]
which holds by definition of $\tilde{\mu}$, $\phi^*$ and the
$\bar{F}$-invariance of $\bar{\mu}$.  For statement
\eqref{properties-bowen-measure-4-discrete-times-cont} let
$\bar{\phi}:\bar{\Lambda}\rightarrow\R$ be defined by
$\bar{\phi}(x)=\phi(z)$, where $z\in\bar\pi^{-1}(x)$.
For any
 $k,l$ positive integers observe
that
\[
\int(\psi.\phi\circ F^k)^*d\bar{\mu}=\int(\psi\circ F^k)^*(\phi\circ
F^k)^*d\bar{\mu}
\]
and
\begin{equation*}
\begin{split}
\left|\int(\psi\phi\circ F^{k+l})^*d\bar{\mu}-\int(\psi\phi\circ
F^k)^*d\bar{\mu}\right|&= \left|\int(\psi\circ
F^{k+l})^*\bar\phi\circ\bar F^{k+l}d\bar{\mu}-\int(\psi\circ
F^k)^*\bar\phi\circ\bar{F}^kd\bar{\mu}\right|\\
&\leq \int\left|(\psi\circ F^{k+l})^*-(\psi\circ
F^k)^*\circ\bar{F}^l\right||\phi\circ\bar F^{k+l}|d\bar{\mu}\\
&\leq \var\psi(k)\|\phi\|_1.
\end{split}
\end{equation*}
Inequality \eqref{properties-bowen-measure-4-discrete-times-cont}
follows letting $l$ go to $\infty$.

We are then left to verify the absolute continuity. While
the properties proved above are intrinsic to the lifting
technique, the disintegration into absolutely continuous conditional
measures on unstable leaves depends on the definition of the
reference measure $\bar{m}$ and the fact that
$\bar{\mu}=\bar{\rho}\bar{m}$. Fix an unstable leaf
$\gamma^u\in\Gamma^u$. Denote by $\lambda_{\gamma^u}$ the conditional Lebesgue measure
on $\gamma^u$. Consider a set
$E\subset\gamma^u$ such that $\lambda_{\gamma^u}(E)=0$. We will show
that $\tilde{\mu}_{\gamma^u}(E)=0$, where $\tilde{\mu}_{\gamma^u}$
denotes the conditional measure of
$\tilde{\mu}$ on $\gamma^u$, except for a few choices 
of $\gamma^u$. To be more precise, the family of curves $\Gamma^u$
induces a partition of $\Lambda$ into unstable leaves which we
denote by~$\L$. Let $\pi_\L:\Lambda\rightarrow \L$ be the natural
projection on the quotient space $\L$, i.e. $\pi_\L(z)=\gamma^u(z)$.
We say that $Q\subset\L$ is measurable if and only if
$\pi_\L^{-1}(Q)$ is measurable.  Let
$\hat{\mu}=(\pi_\L)_*(\tilde{\mu})$, which means that
$\hat{\mu}(Q)=\tilde{\mu}\left(\pi_\L^{-1}(Q)\right)$. We assume
that by definition of $\Gamma^u$ there is a non-decreasing sequence
of finite partitions $\L_1\prec\L_2\prec\ldots\prec\L_n\prec\ldots$
such that $\L=\bigvee_{i=1}^\infty\L_n$. Thus, by Rokhlin
disintegration theorem (see \cite[Appendix C.6]{BDV05}) there is a system
$\left(\tilde{\mu}_{\gamma^u}\right)_{\gamma^u\in\L}$ of conditional
probability measures of $\tilde{\mu}$ with respect to $\L$ such that
\begin{itemize}
\item $\tilde{\mu}_{\gamma^u}(\gamma^u)=1$ for $\hat{\mu}$- almost
every $\gamma^u\in\L$;

\item given any bounded measurable map
$\phi:\Lambda\rightarrow\R$, the map $\gamma^u\mapsto \int\phi
d\tilde{\mu}_{\gamma^u}$ is measurable and $\int\phi
d\tilde{\mu}=\int\left(\int\phi
d\tilde{\mu}_{\gamma^u}\right)d\hat{\mu}$.
\end{itemize}

Let $\bar{E}=\bar{\pi}(E)$. Since the reference measure $\bar{m}$
has a representative $m_{\gamma^u}$ on $\gamma^u$ which is
equivalent to $\lambda_{\gamma^u}$, we have $m_{\gamma^u}(E)=0$ and
$\bar{m}(\bar{E})=0$. As $\bar{\mu}=\bar{\rho} \bar{m}$, then
$\bar{\mu}(\bar{E})=0$. Let
$\bar{\phi}_n:\bar{\Lambda}\rightarrow\R$ be a sequence of
continuous functions such that $\bar{\phi}_n\rightarrow
\I_{\bar{E}}$ as $n\rightarrow\infty$. Consider also the sequence of
continuous functions $\phi_n:\Lambda\rightarrow\R$ given by
$\phi_n=\bar{\phi}_n\circ\bar{\pi}$. Clearly $\phi_n$ is constant in
each $\gamma^s$ stable leaf and
$\phi_n\rightarrow\I_{\bar{E}}\circ\bar{\pi}=
\I_{\bar{\pi}^{-1}(\bar{E})}$ as $n\rightarrow\infty$. By Lebesgue
dominated convergence theorem we have
$\int\phi_nd\tilde{\mu}\rightarrow\int
\I_{\bar{\pi}^{-1}(\bar{E})}d\tilde{\mu}=
\tilde{\mu}\left(\bar{\pi}^{-1}(\bar{E})\right)$ and
$\int\bar{\phi}_nd\bar{\mu}\rightarrow \int\I_{\bar{E}}d\bar{\mu}=
\bar{\mu}(\bar{E})=0$. By
\eqref{properties-bowen-measure-3-discrete} we have
$\int\phi_n\tilde{\mu}=\int\bar{\phi}_nd\bar{\mu}$. Hence, we must
have $\tilde{\mu}\left(\bar{\pi}^{-1}(\bar{E})\right)=0$.
Consequently,
\[
0=\int\I_{\bar{\pi}^{-1}(\bar{E})}d\tilde{\mu}=
\int\left(\int\I_{\bar{\pi}^{-1}(\bar{E})}
d\tilde{\mu}_{\gamma^u}\right)d\hat{\mu}(\gamma^u),
\]
which implies that
$\tilde{\mu}_{\gamma^u}\left(\bar{\pi}^{-1}(\bar{E})\cap\gamma^u\right)=0$
for $\hat{\mu}$-almost every $\gamma^u$.
\end{proof}

\begin{remark}
\label{rem:propeties-bowen-L1} Since the continuous functions are
dense in $L^1$, properties
\eqref{properties-bowen-measure-3-discrete} and
\eqref{properties-bowen-measure-4-discrete-times-cont} also hold when $\phi\in L^1$, by dominated convergence.
\end{remark}

\subsection{Entropy formula}
\label{subsec:stat-stab-saturation-measures} Let $\tilde{\mu}$ be the
SRB measure for $F$ obtained from
$\bar{\mu}=\bar{\rho}\bar{m}$ as in~\eqref{eq:def-bowen-measure}.
We define the saturation of $\tilde{\mu}$ by
\begin{equation}
\label{eq:saturation-definition} \mu^*=\sum_{l=0}^\infty
f^l_*\left(\tilde{\mu}|\{R>l\}\right).
\end{equation}
It is well known that $\mu^*$ is $f$-invariant and that the
finiteness of $\mu^*$ is equivalent to $\int R\,d\tilde{\mu}=\int
R\,d\bar{\mu}<~\infty$. By construction of and $\bar{m}$ and $\bar{\mu}$, the finiteness of $\mu^*$ is also
equivalent to $\int_{\gamma\cap\Lambda} R\,d\leb_\gamma<~\infty$.
Clearly, each $f^l_*\left(\tilde{\mu}|\{R>l\}\right)$  has
absolutely continuous conditional measures on $\{f^l\gamma^u\}$,
which are Pesin unstable manifolds.
Consequently $$\mu=\frac1{\mu^*(M)}\mu^*$$ is an SRB
measure for $f$.

\begin{lemma} \label{lem:lyapunov-rel} If $\lambda$ is a Lyapunov exponent of $\tilde\mu$,
 then $\lambda/\sigma$ is a Lyapunov
exponent of $\mu$, where $\sigma=\int_\Lambda R d\tilde\mu$.
 \end{lemma}

 \begin{proof}As $\mu$ is obtained by saturating $\tilde\mu$ in~\eqref{eq:saturation-definition}, one easily gets
$\mu^*(\Lambda)\geq\tilde\mu(\Lambda)=1$, and so $\mu(\Lambda)>0$.
By ergodicity, it is enough to compare the Lyapunov exponents for
points $z\in\Lambda$.
 Let $n$ be a positive integer.  We have for each $z\in\Lambda$
  $$F^n(z)=f^{S_n(z)}(z),\quad\text{where $S_n(z)=
  \sum_{i=0}^{n-1}R(F^i(z))$}.$$
 As $S_n(z) = S_n(\zeta)$ for Lebesgue almost every $z\in \Lambda$
 and $\zeta$
 close to $z$, we have for
 $v\in T_zM$
  \begin{equation}\label{eq.ly}
  \frac{1}{S_n(z)}\log\|Df^{S_n(z)}(z)v\|=
  \frac{n}{nS_n(z)}\log\|DF^{n}(z)v\|.
  \end{equation}
Since $\tilde\mu$ is ergodic, Birkhoff ergodic theorem yields
 \begin{equation}\label{eq.bi}
 \lim_{n\to\infty}\frac{S_n(z)}{n}=\int_\Lambda R\,d\tilde\mu=\sigma
 \end{equation}
for $\tilde\mu$ almost every $z\in\Lambda$.
\end{proof}

\begin{proposition}\label{prop:entropy-formula}
Let $J\bar F$ be the
Jacobian of $\bar F$ with respect to the
measure $\bar m$ on $\bar\Lambda$. Then
\begin{equation*}
\label{eq:entropy-formula} h_{\mu}=\sigma^{-1}\int_{\bar\Lambda}\log
J\bar Fd\bar m.
\end{equation*}
\end{proposition}
\begin{proof}
By \cite[Corollary~7.4.2]{LY2}
we have
\begin{equation}
\label{eq:entropy-formula-1}
h_{\mu}=\sum_{\lambda_i>0}\lambda_i\dim E_i,
\end{equation}
where $\lambda_i$ are Lyapunov exponents of $\mu$ and $E_i$ the
corresponding  linear spaces given by Oseledets' decomposition.
By Lemma~\ref{lem:lyapunov-rel} we have
\[
h_{\mu}=\sigma^{-1}\sum_{\tilde\lambda_i>0}\tilde\lambda_i\dim E_i,
\]
where $\tilde\lambda_i$ are Lyapunov exponents of $\tilde\mu$. As a consequence of Oseledets theorem
 we may also write
\[
\sum_{\tilde\lambda_i>0}\tilde\lambda_i\dim E_i=\int_\Lambda\log\det D
F^u d\tilde\mu.
\]
According to \eqref{eq:jacobians-rel},
\begin{equation*}
\begin{split}
\int_\Lambda\log JF d\tilde\mu&=\int_\Lambda\log\det D F^u d\tilde\mu+
\int_\Lambda\log \hat u\circ F d\tilde\mu-\int_\Lambda\log \hat u d\tilde\mu \\
&=\int_\Lambda\log\det D F^u d\tilde\mu,
\end{split}
\end{equation*}
where the last equality follows from the $F$-invariance of $\tilde\mu$.
Finally, since by Lemma~\ref{lem:jac-def} $JF$ is constant in each
$\gamma^s$-leaf it follows from
Proposition~\ref{prop:properties-bowen-measure}~
\eqref{properties-bowen-measure-3-discrete} that
\[
\int_\Lambda\log JF d\tilde\mu=\int_{\bar\Lambda}\log J\bar Fd\bar
m.
\]
\end{proof}

\section{Statistical Stability}\label{se:tres}

Let $\F$ be a uniform family of maps. Fix $f_0\in\F$ and take any
sequence $(f_n)_{n\geq1}$ in $\F$ such that $f_n\to f_0$, as
$n\to\infty$, in the $C^k$ topology. For each $n\ge0$, let $\mu_n$ denote the (unique) SRB measure for $f_n$. Given $n\ge0$, the map
$f_n\in\F$ admits a Gibbs-Markov structure $\Lambda_n$  with $\Gamma^u_n=\{\gamma^u_n\}$
and $\Gamma^s_n=\{\gamma^s_n\}$ its defining families of unstable
and stable leaves. Consider $R_n:\Lambda_n\to\N$ the
return time,
$F_n:\Lambda_n\rightarrow\Lambda_n$ the induced map,
 $\hat\gamma_n$ the special unstable leaf given
by condition \eqref{u-prox-unstable-direction} and
$H_n:\hat\gamma_n\cap\Gamma^s_0\to\hat\gamma_0$  obtained by sliding through the
stable leaves of~$\Lambda_0$. Recall that
$\Omega^n_0=H_n(\hat\gamma_n\cap\Lambda_n)$ and
$\Omega_0=\hat\gamma_0\cap\Lambda_0$.

\begin{remark}
\label{rem:iterates-cont} Since $f_n\to f_0$, as $n\to\infty$, in
the $C^k$ topology and \eqref{u-prox-unstable-direction} holds, then
for every $\varepsilon>0$ and $\ell\in\N$, there exists $N_0\in\N$
such that for every $n\geq N_0$ we have
$$\|\hat\gamma_n-\hat\gamma_0\|_1<\varepsilon,$$
\[
\max_{x\in\Omega_0\cap\Omega^n_0}
\left\{|(f_n\circ H_n^{-1}-f_0)(x)|,\ldots,|(f_n^\ell\circ
H_n^{-1}-f_0^\ell)(x)|\right\}<\varepsilon,
\]
and
\[
\max_{x\in\Omega_0\cap\Omega^n_0}
\left\{\left|\log\frac{\det Df_n^u(f_n\circ H_n^{-1}(x))}{\det
        Df_0^u(f_0(x))}\right|,\ldots,\left|\log\frac{\det Df_n^u(f_n^\ell\circ H_n^{-1}(x))}{\det
        Df_0^u(f_0^\ell(x))}\right|\right\}<\varepsilon.
\]
\end{remark}

 Our goal is to show that $\mu_n\rightarrow\mu_0$
in the weak* topology, i.e. for each continuous function
$g:M\rightarrow\R$ the sequence $\int g\,d\mu_n$  converges to $\int
g \,d\mu_0$. We will show that given any continuous
$g:M\rightarrow\R$, each subsequence of $\int g\,d\mu_n$ admits a
subsequence converging  to $\int g\, d\mu_0$.

\subsection{Convergence of  the densities on the reference leaf}
\label{subsec:cont-quotient}
In Section \ref{subsec:quotient-natural-measure} we built a family
of  holonomy invariant measures on unstable leaves that gives rise
to a measure $\bar m_n$ on $\bar{\Lambda}_n$. Moreover,
\begin{equation}\label{eq:pi-chapeu}(\hat\pi_n )_*
m_{\hat\gamma_n}=\bar m_n\quad\text{and}\quad m_{\hat\gamma_n}=
\I_{\hat\gamma_n\cap\Lambda_n}\ln
 ,
 \end{equation} where $\I_{(\cdot)}$ stands for the indicator function.
By Lemma~\ref{lem:small-adjust}, for each $n\ge0$ there is an
$\bar{F}_n$-invariant measure $\bar\mu_n=\bar{\rho}_n \bar m_n$ with  $\|\bar\rho_n\|_\infty\leq K$ for
all $n\ge0$. We define the sequence $(\varrho_n)_{n\ge0}$ of functions in~$\hat\gamma_0$ as
  \begin{equation}
  \label{eq:def-varrho}\varrho_n=\bar\rho_n\circ\hat\pi_n\circ H_n^{-1}
  \cdot\I_{\Omega^n_0},\end{equation}
which in particular gives
 $$\varrho_0=\bar\rho_0\circ\hat\pi_0.$$

 The main purpose of this section is to prove that \emph{the sequence $\left({\varrho}_{n}\right)_{n\in\N}$ converges to
${\varrho}_{0}$ in the weak*
topology}.
By Banach-Alaoglu theorem there is a subsequence
$\left({\varrho}_{n_i}\right)_{i\in\N}$ converging to some
${\varrho}_{\infty}\in
L^{\infty}(\l0)$ in the weak*
topology, i.e.
\begin{equation}
\label{eq:banach-alaoglu-sequence}
\int\phi{\varrho}_{n_i}d\leb_{\hat\gamma_0}
\xrightarrow[i\rightarrow\infty]{}
 \int\phi{\varrho}_\infty d\leb_{\hat\gamma_0},\quad
\forall \phi\in L^1(\l0).
\end{equation}
The following lemma establishes that integration with respect to $\bar m_n$ is close to integration with respect to $\varrho_n\l0$, up to a small error.

\begin{lemma}
\label{lem:rel-integration-unstab-leaves}
Let $\bar\phi\in L^\infty(\bar m_n)$. If $n$ is sufficiently large, then
\begin{equation*}
\left|\int_{\bar\Lambda_n}\bar\phi \bar\rho_n\;d\bar m_n-
\int_{\Omega^n_0}(\bar\phi \circ\hat\pi_n\circ
H_n^{-1})\varrho_n\,d\l0\right|\leq
K\|\bar\phi\|_\infty Q_n,
\end{equation*}
where
$Q_n=\ln(\Omega_n^0\bigtriangleup\Omega_n)+\left|\int_{\Omega_{0}^n} d(H_{n})_*\ln -
\int_{\Omega_{0}^n} d\l0\right|$.
\end{lemma}

\begin{proof}
  By \eqref{eq:pi-chapeu}, we have
  $\int_{\bar\Lambda_n}\bar\phi \bar\rho_n\;d\bar m_n
  =\int_{\Omega_n}(\bar\phi\circ\hat\pi_n)( \bar\rho_n\circ\hat\pi_n)\;
  d\ln.$ It follows that
  \begin{align*}
    \left|\int_{\bar\Lambda_n}\bar\phi \bar\rho_n\;d\bar m_n
    \right.&\left.-\int_{\Omega^n_0}(\bar\phi
\circ\hat\pi_n\circ
H_n^{-1})\varrho_n\,d\l0\right|\leq\left|\int_{\Omega_n^0
\bigtriangleup \Omega_n}(\bar\phi\circ\hat\pi_n)(
\bar\rho_n\circ\hat\pi_n)\;d\ln\right|\\
  &\quad+\left|\int_{\Omega_n^0
\cap\Omega_n}(\bar\phi\circ\hat\pi_n)( \bar\rho_n\circ\hat\pi_n)\;
  d\ln-\int_{\Omega^n_0}(\bar\phi
\circ\hat\pi_n\circ H_n^{-1})\varrho_n\,d\l0\right|\\
&\leq
K\|\bar\phi\|_\infty\ln(\Omega_n^0\bigtriangleup\Omega_n)\\
&\quad+\left|\int_{\Omega^n_0}(\bar\phi\circ\hat\pi_n\circ
H_n^{-1})\varrho_n\;
  d(H_n)_*\ln-\int_{\Omega^n_0}(\bar\phi
\circ\hat\pi_n\circ H_n^{-1})\varrho_n\,d\l0\right|\\
&\leq K\|\bar\phi\|_\infty\ln(\Omega_n^0\bigtriangleup\Omega_n)+
K\|\bar\phi\|_\infty\left|\int_{\Omega_{0}^n} d(H_{n})_*\ln -
\int_{\Omega_{0}^n} d\l0\right|.
  \end{align*}

\end{proof}

Consider the maps $G_0:\hat\gamma_0\to\hat\gamma_0$ and $G_{n}:\hat\gamma_0\to\hat\gamma_n$
defined by $$G_{0}=\hat\pi_{0}^{-1}\circ\bar F_{0}\circ\hat\pi_{0}
\quad\text{and}\quad G_{n}=\hat\pi_{n}^{-1}\circ\bar F_{n}\circ\hat\pi_{n}\circ H_{n}^{-1}.$$

\begin{lemma}
\label{lem:bar-F0-bar-Fni-proximity} For every $\varepsilon>0$,
$n\in\N$ sufficiently large and $\leb_{\hat\gamma_0}$-almost every
$x\in\Omega_0\cap\Omega_0^n\cap\{R_{n}=\ell\}\cap\{R_0=\ell\}$
we have $|G_{n}(x)-G_0(x)|<\varepsilon$.
\end{lemma}

\begin{proof}
Consider a point
$x\in\Omega_0\cap \Omega_0^n\cap\{R_{n}=\ell\}\cap\{R_0=\ell\}$.
We may assume that $G_n(x)$ 
is a Lebesgue density point of $\Omega_n$. Then, using
\eqref{u-matching-cantor-sets} and the continuity of the stable
foliation (see
Definition~\ref{def:cont-families-of-(un)stable-curves}~(iii)), for
sufficiently large $n\in\N$ we may guarantee the existence of a
point $\tilde y\in \Omega_n^0\cap\Omega_{n}$ such that
$\gamma_n^s(\tilde y)$ is at most $\varepsilon\sin(\theta)/4$ apart
from $\gamma_n^s(G_n(x))$ in the $C^1$-norm; recall Remark~\ref{rem:Leb-Hm*ln-l0}.
Using \eqref{u-proximity-stable-direction} we may assume that
$n\in\N$ is also sufficiently large so that the distance in the
$C^1$ norm between $\gamma_{n}^s(\tilde y)$ and $\gamma_0^s(\tilde
y) $  is at most $\varepsilon\sin(\theta)/4$.

Taking into account Remark~\ref{rem:iterates-cont} and the
continuity of the stable foliation, we may assume that $n\in\N$ is large
enough so that $|f_{n}^l(H_{n}^{-1}(x))-f_0^l(x)|$ is sufficiently
small in order to $\gamma_{0}^s(f_{0}^l(x))$ belong to a
$\varepsilon\sin(\theta)/4$-neighborhood of $\gamma_{0}^s(\tilde
y)$, in the $C^1$-norm. It follows that
$\gamma_n^s(f_n^l(H_n^{-1}(x)))$ and $\gamma_{0}^s(f_{0}^l(x))$ are
at most $3\varepsilon\sin(\theta)/4$ apart, in the $C^1$-norm.
Finally, observing  that $G_n(x)=\gamma_n^s(f_n^l(H_n^{-1}(x)))\cap
\gamma^u_n$, $G_0(x)=\gamma_{0}^s(f_{0}^l(x))\cap\gamma^u_0$ and
$\gamma^u_n$ can be made
 arbitrarily close to $\gamma^u_0$, in the $C^1$-norm
  (by \eqref{u-prox-unstable-direction}),
 then, as long as $n$ is sufficiently large, we have
 $|G_n(x)-G_0(x)|<\varepsilon $.
\begin{figure}[h]
\includegraphics[scale=0.9]{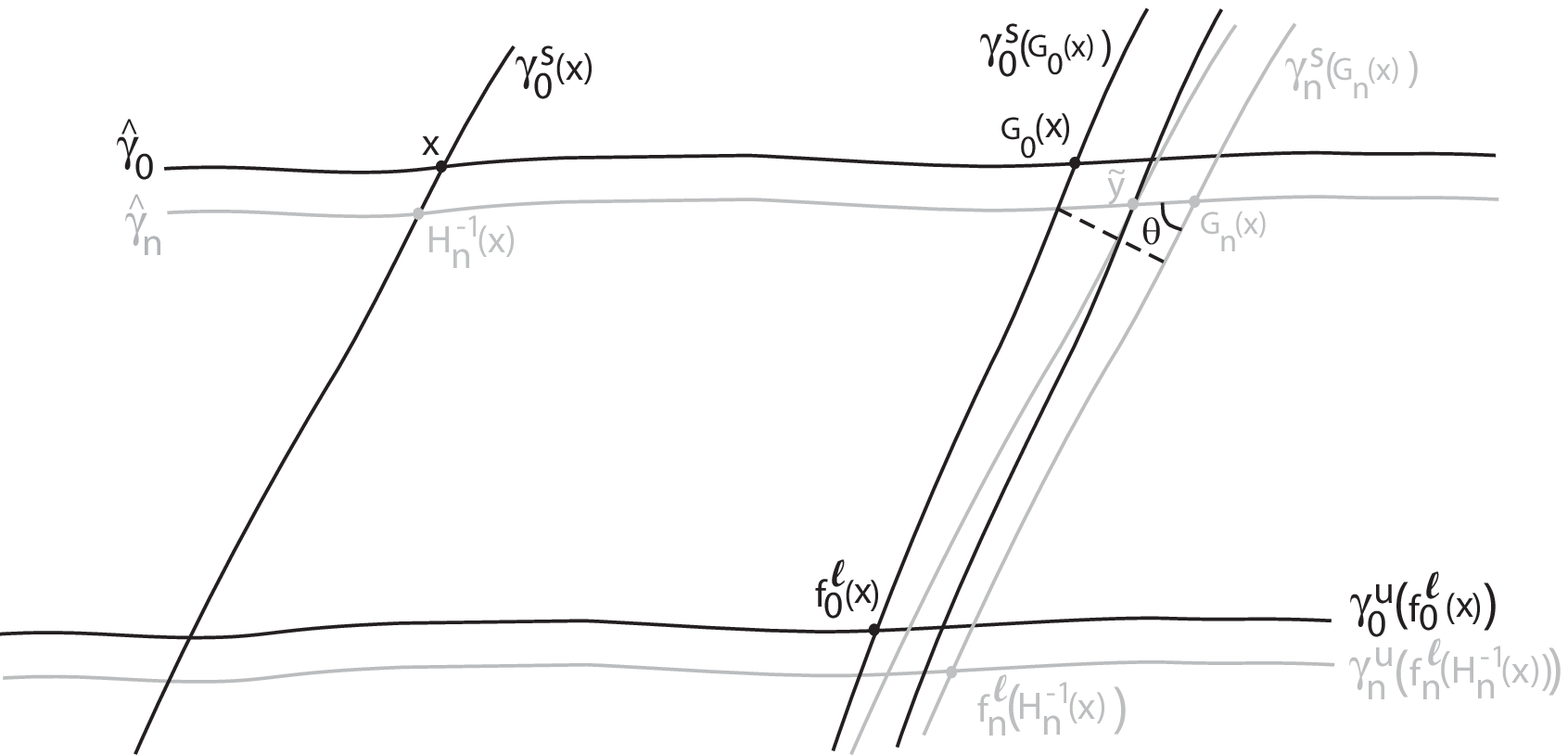}
\caption{} \label{fig:F0invariance}
\end{figure}
\end{proof}

\begin{proposition}
\label{lem:rho-infty-invariance} The measure
$
({\varrho}_\infty \circ \hat\pi_0^{-1})\bar{m}_0$ is
$\bar{F}_0$-invariant.
\end{proposition}

\begin{proof}
We just have to verify that for every continuous
$\varphi:\bar{\Lambda}_0\rightarrow \R$
\begin{equation*}\label{eq.fifi}
\int(\varphi\circ\bar{F}_0)({\varrho}_\infty \circ \hat\pi_0^{-1}) d {\bar m}_0=\int
\varphi ({\varrho}_\infty \circ \hat\pi_0^{-1}) d{\bar m}_0
\end{equation*}
Given such $\varphi$, consider a continuous function
$\phi:M\rightarrow\R$ such that $\|\phi\|_\infty\leq\|\varphi\|_\infty$ and
$\phi|_{\Omega_0}=\varphi\circ\hat\pi_0$.
Since $\bar{\mu}_{n_i}=\bar\rho_{n_i}d\bar{m}_{n_i}$ is
$\bar{F}_{n_i}$-invariant we have
\begin{equation}
  \label{eq:lema-rho-infty-rhoni-invariance}
  \int (\phi \circ \hat\pi_{n_i}^{-1}\circ \bar F_{n_i})
  \bar\rho_{n_i}d\bar{m}_{n_i}=\int (\phi \circ \hat\pi_{n_i}^{-1})
  \bar\rho_{n_i}d\bar{m}_{n_i}
\end{equation}
Recalling definitions \eqref{eq:pi-chapeu},\eqref{eq:def-varrho},
the fact that $\varrho_{n_i}$ is supported on
$\Omega_0^{n_i}\subset\Omega_0$ and applying
Lemmas~\ref{lem:rel-integration-unstab-leaves} and
\ref{lem:small-adjust} we get
\begin{align*}
\left|\int (\phi \circ \hat\pi_{n_i}^{-1})\bar\rho_{n_i}\right.
&\left.d\bar{m}_{n_i}- \int
\varphi({\varrho}_\infty \circ \hat\pi_0^{-1})\; d{\bar m}_0\right|\leq \\
&\leq \left|\int(\phi \circ H_{n_i}^{-1})\varrho_{n_i}d\l0 -
\int(\varphi\circ\hat\pi_0)\varrho_\infty d\l0 \right|+Q_{n_i}\\
&=\left|\int(\phi \circ H_{n_i}^{-1})\varrho_{n_i}d\l0 -\int\phi
\varrho_\infty d\l0 \right|+Q_{n_i}\\
&\leq\left|\int(\phi \circ H_{n_i}^{-1})\varrho_{n_i}d\l0 -\int\phi
\varrho_{n_i}d\l0 \right|+\\
&\quad +\left|\int\phi\varrho_{n_i}d\l0 -\int\phi\varrho_\infty d\l0
\right|+Q_{n_i}\\
&\leq K\int\left|\phi \circ H_{n_i}^{-1}-\phi\right|d\l0
+\left|\int\phi\varrho_{n_i}d\l0 -\int\phi\varrho_\infty d\l0
\right|+Q_{n_i}\\
\end{align*}
Therefore, using \eqref{u-prox-unstable-direction} for the first
term on the right, \eqref{eq:banach-alaoglu-sequence} for the second
and \eqref{u-matching-cantor-sets} plus
Remark~\ref{rem:Leb-Hm*ln-l0} for the $Q$ term, we conclude that
\begin{equation}
\label{eq:convergence-1-lema-rho-infty} \int (\phi \circ \hat\pi_{n_i}^{-1})\bar\rho_{n_i}d\bar{m}_{n_i}
\xrightarrow[i\rightarrow\infty]{}\int
\varphi ({\varrho}_\infty \circ \hat\pi_0^{-1}) d{\bar m}_0.
\end{equation}
Once we prove the next claim, then equality
\eqref{eq:lema-rho-infty-rhoni-invariance}, the limit \eqref{eq:convergence-1-lema-rho-infty} and the uniqueness of the
limit give the desired result.
\begin{claim}
\label{lem:estimate-for-rho-infty-invariance}
$\displaystyle\int (\phi \circ \hat\pi_{n_i}^{-1}\circ \bar F_{n_i})\bar\rho_{n_i}d\bar{m}_{n_i}
\xrightarrow[i\rightarrow\infty]{}\int
\varphi\circ \bar F_{0} ({\varrho}_\infty \circ \hat\pi_0^{-1}) d{\bar m}_0.$
\end{claim}
Let
\[
E_1:=\left|
\int (\phi \circ \hat\pi_{n_i}^{-1}\circ \bar F_{n_i})\bar\rho_{n_i}d\bar{m}_{n_i}
-\int
\varphi\circ \bar F_{0} ({\varrho}_\infty \circ \hat\pi_0^{-1}) d{\bar m}_0\right|.
\]
Again, using definitions \eqref{eq:pi-chapeu},\eqref{eq:def-varrho}
and applying Lemma~\ref{lem:rel-integration-unstab-leaves} we get
\[
E_1\leq \left|
\int (\phi \circ G_{n_i})\varrho_{n_i}d\l0
-\int
(\phi\circ G_0) {\varrho}_\infty d\l0\right|+Q_{n_i}
\]
Now, observe that by \eqref{u-matching-cantor-sets} and
Remark~\ref{rem:Leb-Hm*ln-l0} the term  $Q_{n_i}$ can be made
arbitrarily small for large $i$. This leaves us with the first term
on the right that we denotee by $E_2$.  Using
Lemma~\ref{lem:small-adjust} we have
\begin{align*}
E_2
&\leq\int\left|\phi \circ G_{n_i}-\phi\circ G_0\right|\varrho_{n_i}d\l0
+\left|\int(\phi\circ G_0)\varrho_{n_i}d\l0 -\int(\phi\circ G_0)\varrho_\infty d\l0
\right|\\
&\leq K\int\left|\phi \circ G_{n_i}-\phi\circ G_0\right|d\l0
+\left|\int(\phi\circ G_0)\varrho_{n_i}d\l0 -\int(\phi\circ G_0)\varrho_\infty d\l0
\right|\\
\end{align*}
According to equation \eqref{eq:banach-alaoglu-sequence} it is clear
that the last term on the right can be made arbitrarily small
provided $i$ is large enough.  So, denote by $E_3$ the first term on
the right. Recalling the fact that $\varrho_{n_i}$ is supported on
$\Omega_0^{n_i}\subset\Omega_0$, we have for any $N$
\begin{equation*}
\begin{split}
E_3 &\leq
K\|\phi\|_\infty\sum_{\ell=N+1}^{\infty}\left(\l0(\{R_{n_i}=\ell\})+
\l0(\{R_0=\ell\})\right)\\&\quad+
K\|\phi\|_\infty\sum_{\ell=1}^{N}\l0(\{R_{n_i}=\ell\}
\bigtriangleup\{R_0=\ell\})
\\
&\quad
+K\sum_{\ell=1}^{N}\int_{\{R_{n_i}=\ell\}\cap\{R_0=\ell\}\cap\Omega_0
\cap\Omega_0^{n_i}} \left|\phi\circ G_{n_i}-\phi\circ G_0\right|
d\leb_{\hat\gamma_{0}} .
\end{split}
\end{equation*}
Denote by $E_4$, $E_5$ and $E_6$ respectively the terms in the last
sum. Having in mind \eqref{u-uniform-tail} and
Remark~\ref{rem:Leb-Hm*ln-l0}, we may choose $N\in\N$ sufficiently
large so that $E_4$ is small for large $i$. For this choice of $N$,
by \eqref{u-matching-s-sublattices}, we also have that $E_5$ is
small for large $i$. We now turn our attention to $E_6$. For
$\ell=1,\ldots,N$, let
$$E_6^\ell=\int_{\{R_{n_i}=\ell\}\cap\{R_0=\ell\}}
\left|\phi\circ G_{n_i}-\phi\circ G_0\right|
\I_{\Omega_0\cap\Omega_0^{n_i}}d\leb_{\hat\gamma_{0}} .$$ Since
$\phi$ is continuous and $M$ is compact then each $E_6^\ell$ can be
made arbitrarily small by Lemma~\ref{lem:bar-F0-bar-Fni-proximity}.
\end{proof}

\begin{corollary}\label{co:conv-densities-weak*}
Given $\phi\in L^1(\leb_{\hat\gamma_0})$, we have
\begin{equation*}
\int\phi\varrho_{n}d\leb_{\hat\gamma_0}
\xrightarrow[n\rightarrow\infty]{}
 \int\phi\varrho_0 d\leb_{\hat\gamma_0}.\quad
\end{equation*}
\end{corollary}
\begin{proof}
By uniqueness of the absolutely continuous invariant measure for $\bar F$, it follows from Proposition~\ref{lem:rho-infty-invariance} that
$\bar\rho_0=\varrho_\infty\circ\hat\pi_0^{-1}$, which immediately yields $\varrho_\infty=\varrho_0$.  Hence
\begin{equation}
\label{eq:conv-densities-weak*}
\int\phi\varrho_{n_i}d\leb_{\hat\gamma_0}
\xrightarrow[i\rightarrow\infty]{}
 \int\phi\varrho_0 d\leb_{\hat\gamma_0},\quad
\text{for all $\phi$ continuous}.
\end{equation}
The same argument proves that any subsequence of $(\varrho_{n})_n$ has a weak* convergent subsequence with limit also equal to $\varrho_{0}$. This shows that $(\varrho_{n})_n$ itself converges to $\varrho_{0}$ in the weak* topology. Since continuous functions are dense in $L^1(\leb_{\hat\gamma_0})$, using that the densities $\varrho_n$ are uniformly bounded, by Lemma~\ref{lem:small-adjust}, the result follows easily from \eqref{eq:conv-densities-weak*}.
\end{proof}

\subsection{Continuity of the SRB measures}
For each $n\ge0$ let
$\tilde{\mu}_{n}$ be the  $F_{n}$- invariant measure lifted from $\bar{\mu}_{n}$
as in
\eqref{eq:def-bowen-measure},
$\mu^*_{n}$ the saturation of
$\tilde{\mu}_{n}$ as in
 \eqref{eq:saturation-definition},  and
$\mu_{n}=\mu^*_{n}/\mu^*_{n}(M)$ the SRB measure.
The main goal of this section is to prove the following result.
\begin{proposition}
\label{prop:convergence-SRB-measures} For every continuous
$g:M\rightarrow\R$,
\[
\int gd\mu^*_{n}\xrightarrow[i\rightarrow\infty]{}\int gd\mu^*_0.
\]
\end{proposition}
\begin{proof}
As $M$ is compact, then $g$
is uniformly continuous and $\|g\|_\infty<\infty$.
Recalling  \eqref{eq:saturation-definition} we may write for all
$n\in\N_0$ and every integer $N_0$
\[
\mu^*_n=\sum_{\ell=0}^{N_0-1}\mu_n^\ell+\eta_n,
\] where
$\mu_n^\ell=f^\ell_*(\tilde{\mu_n}|\{R_n>\ell\})$ and
$\eta_n=\sum_{\ell\ge N_0} f^\ell_*(\tilde{\mu_n}|\{R_n>l\})$. By
\eqref{u-uniform-tail}, we may choose $N_0$ so that $\eta_n(M)$ is
as small as we want, for all $n\in\N_0$.
We are left to show that for every $\ell<N_0$, if $n$ is large enough then
\[
 \left| \int (g\circ f_{n}^\ell)
 \I_{\{R_{n}>\ell\}}d\tilde{\mu}_{n}-\int (g\circ f_{0}^\ell)
\I_{\{R_{0}>\ell\}}d\tilde{\mu}_{0}\right|
\] is arbitrarily small.
 We fix $\ell<N_0$ and take $k\in\N$ large so that
$\var(g(k))$ is sufficiently small. Then, we use Proposition
\ref{prop:properties-bowen-measure}
\eqref{properties-bowen-measure-4-discrete-times-cont} and its
Remark \ref{rem:propeties-bowen-L1} to reduce our problem to
controlling the following error term:
\[
E:=\left| \int (g\circ f_{n}^\ell\circ F_{n}^k)^*
(\I_{\{R_{n}>\ell\}}\circ F_{n}^k)^*d\bar{\mu}_{n}- \int (g\circ
f_{0}^\ell \circ F_{0}^k)^* (\I_{\{R_{0}>\ell\}}\circ F_{0}^k)^*
d\bar{\mu}_{0}\right|.
\]
Let $\varrho_0:\hat\gamma_0\to\R$ be such that $\varrho_0=\bar\rho_0\circ\hat\pi_0\cdot\I_{\Omega_0}$
and define
\begin{multline*}
E_0=\left| \int \left(\left((g\circ f_{n}^\ell \circ
F_{n}^k)^\bullet
(\I_{\{R_{n}>\ell \}}\circ F_{n}^k)^\bullet \right)\, \circ H_n^{-1}\right)
\varrho_n d\l0\right.\\
-\left. \int \left((g\circ f_{0}^\ell  \circ F_{0}^k)^\bullet
(\I_{\{R_{0}>\ell \}}\circ F_{0}^k)^\bullet \right)
\varrho_0\;d\l0\right|.
\end{multline*}
By Lemma~\ref{lem:rel-integration-unstab-leaves}, we have $E\leq
E_0+K\|g\|_\infty Q_n$.
Observe that by \eqref{u-matching-cantor-sets} and
Remark~\ref{rem:Leb-Hm*ln-l0} we may consider $n$ large
enough so that
$ K\|g\|_\infty Q_n$ is negligible.
Applying the triangular inequality we get
\begin{equation*}
\begin{split}
E_0 &\leq K \int\left|(g\circ f_{n}^\ell \circ F_{n}^k)^\bullet
\circ H_n^{-1} -(g\circ f_{0}^\ell  \circ
F_{0}^k)^\bullet \right|\I_{\Omega_0\cap\Omega_0^{n}}d\l0\\
&\quad +K\|g\|_\infty\int\left|(\I_{\{R_{n}>\ell \}}\circ
F_{n}^k)^\bullet  \circ H_n^{-1}-(\I_{\{R_{0}>\ell \}}\circ
F_{0}^k)^\bullet \right|\I_{\Omega_0\cap\Omega_0^{n}}d\l0
\\&\quad +\left|\int (g\circ f_{0}^\ell  \circ F_{0}^k)^\bullet
(\I_{\{R_{0}>\ell \}}\circ F_{0}^k)^\bullet \,
\I_{\Omega_0\cap\Omega_0^{n}}\,
\left(\varrho_{n}-\varrho_0\right)d\l0\right|.
\end{split}
\end{equation*}
By Corollary~\ref{co:conv-densities-weak*} the term
 $$
\left|\int (g\circ f_{0}^\ell  \circ F_{0}^k)^\bullet
(\I_{\{R_{0}>\ell \}}\circ F_{0}^k)^\bullet\,
\I_{\Omega_0\cap\Omega_0^{n}}\,
\left(\varrho_{n}-\varrho_0\right)d\l0\right|
$$
is as small as we want as long as $n$ is large enough. The analysis
of  the remaining terms
 $$
 \int\left|(g\circ f_{n}^\ell \circ F_{n}^k)^\bullet
 \circ H_n^{-1} -(g\circ
f_{0}^\ell  \circ F_{0}^k)^\bullet
\right|\I_{\Omega_0\cap\Omega_0^{n}}d\l0
$$
and
 $$
\int\left|(\I_{\{R_{n}>\ell \}}\circ F_{n}^k)^\bullet \circ
H_n^{-1}-(\I_{\{R_{0}>\ell \}}\circ F_{0}^k)^\bullet
\right|\I_{\Omega_0\cap\Omega_0^{n}}d\l0
$$
is left to  Lemmas \ref{lem:control-of-E2} and
\ref{lem:control-of-E3}, respectively.
\end{proof}

In the proofs of Lemmas \ref{lem:control-of-E2} and
\ref{lem:control-of-E3} we have to produce a suitable positive
integer $N$ so that returns that take longer than $N$ iterations are
negligible. The next lemma provides the tools for an adequate
choice. We consider the sequence of consecutive return times for  $z\in\Lambda$
\begin{equation}
  \label{eq:def-consecutive-returns}
  R^1(z)=R(z)\quad\mbox{and}\quad R^n(z)=R\left(f^{R^1+R^2+\ldots+R^{n-1}}(z)\right).
\end{equation}

\begin{lemma}
\label{lem:choice-N5} Given $k,N\in\N$
\[
\bar m\left(\left\{z\in \Lambda: \, \exists
t\in\{1,\ldots,k\}\,\text{such that }R^t(z)>N\right\}\right)\leq k
C_1\,\bar m(\{R>N\}).
\]
\end{lemma}
\begin{proof}
We may write
\[
\left\{z\in \Lambda: \, \exists t\in\{1,\ldots,k\}\,\text{ such that
} R^t(z)>N\right\}=\bigcup_{t=0}^{k-1} B_t,
\]
where
\[
B_t=\left\{z\in \Lambda: \,R(z)\leq N,\ldots,R^{t}(z)\leq N,
R^{t+1}(z)>N\right\}.
\]
If $R(z)\leq N,\ldots,R^{t}(z)\leq N$ then
there exist $j_1,\ldots j_t\leq N$ with $R(\Upsilon_{j_l})\leq N$
for every $l=1,\ldots,t$ and $z\in \Upsilon_{j_1,\ldots,j_t}$.
Observe that $\bar F^t\left(\Upsilon_{j_1,\ldots,j_t}\right)
=\bar\Lambda$ and there
is $y\in\Upsilon_{j_1,\ldots,j_t}$ such that $\bar
m(\bar\Lambda)\leq J\bar F^t(y).\bar m (\Upsilon_{j_1,\ldots,j_t})$.
Also, there exists $x\in\Upsilon_{j_1,\ldots,j_t}\cap\bar
F^{-t}(\{R>N\})$ such that $\bar m(\{R>N\})\geq J\bar F^t(x).\bar
m(\Upsilon_{j_1,\ldots,j_t}\cap\bar F^{-t}(\{R>N\})$. Then, using
bounded distortion we obtain
\[
\frac{\bar m(\Upsilon_{j_1,\ldots,j_t}\cap\bar F^{-t}(\{R>N\})}{\bar
m (\Upsilon_{j_1,\ldots,j_t})}\leq \frac{J\bar F^t(y)}{ J\bar
F^t(x)}\frac{\bar m(\{R>N\})} {\bar m(\bar\Lambda)}\leq C_1\,\bar
m(\{R>N\}),
\]
Finally, we conclude that
\begin{align*}
|B_t|&= \sum_{j_1,\ldots,j_t :\,
            R(\Upsilon_{j_l})\leq N,\,l=1\ldots t}
            \bar m(\Upsilon_{j_1,\ldots,j_t}\cap\bar F^{-t}(\{R>N\})
            \\
            &\leq C_1\,\bar m(\{R>N\})\sum_{j_1,\ldots,j_t :\,
            R(\Upsilon_{j_l})\leq N,\,l=1\ldots t}
            \bar m (\Upsilon_{j_1,\ldots,j_t})\\
            &\leq C_1\,\bar m(\{R>N\}).
\end{align*}
\end{proof}

\begin{lemma}
\label{lem:control-of-E2} Given $\ell,k\in\N$ and  $\varepsilon>0$
there is $J\in\N$ such that for every $n>J$
\[\int\left|(g\circ f_{n}^\ell \circ F_{n}^k)^\bullet
 \circ H_n^{-1} -(g\circ
f_{0}^\ell  \circ F_{0}^k)^\bullet
\right|\I_{\Omega_0\cap\Omega_0^{n}}d\l0<\varepsilon.
\]
\end{lemma}
\begin{proof}
We split the argument into three steps:
\begin{enumerate}
\item \label{item:control-E2-1} We appeal to Lemma \ref{lem:choice-N5}
to choose $N\in\N$ sufficiently large so that the set
$$L:=\left\{x\in\Omega_0\cap\Omega_0^{n}:\,\exists
t\in\{1,\ldots,k\}\, R_0^t(x)>N \,\mbox{or}\,
R_{n}^t(x)>N\right\}$$ has sufficiently small mass.

\item \label{item:control-E2-2} We pick $J\in\N$ large enough to
guarantee that, according to condition
\eqref{u-matching-s-sublattices}, for every $k$ positive integers
$j_1,\ldots,j_k$ such that $R_0(\Upsilon^0_{j_l})\leq N$, for all
$i=1,\ldots,k$, each set $\Upsilon^0_{j_1,\ldots,j_k}$ and its
corresponding $\Upsilon^{n}_{j_1,\ldots,j_k}$ satisfy the condition:
$\Upsilon^0_{j_1,\ldots,j_k}\bigtriangleup
H_{n}\left(\Upsilon^{n}_{j_1,\ldots,j_k}\right)$ has sufficiently
small conditional Lebesgue measure.

\item \label{item:control-E2-3} Finally,
in each set $\Upsilon^0_{j_1,\ldots,j_k}\cap
H_{n}\left(\Upsilon^{n}_{j_1,\ldots,j_k}\right)$ we  control
\[
\left|(g\circ f_{n}^\ell \circ F_{n}^k)^\bullet
 \circ H_n^{-1} -(g\circ
f_{0}^\ell  \circ F_{0}^k)^\bullet \right|.
\]
\end{enumerate}

\noindent Step \eqref{item:control-E2-1}: From Lemma
\ref{lem:choice-N5} we have $|L|\leq  k
C_1.\left(\l0(\{R_0>N\})+\ln(\{R_{n}>N\})\right)$. So, by assumption
\eqref{u-uniform-tail}, we may choose $N$ and $J$ large enough so
that
\[
2\|g\|_\infty k
C_1.\left(\l0(\{R_0>N\})+\ln(\{R_{n}>N\})\right)<\frac{\varepsilon}{3},
\]
which implies that
\[
\int_{L}\left|(g\circ f_{n}^\ell \circ F_{n}^k)^\bullet
 \circ H_n^{-1} -(g\circ
f_{0}^\ell  \circ F_{0}^k)^\bullet
\right|\I_{\Omega_0\cap\Omega_0^{n}}d\l0< \frac{\varepsilon}{3}.
\]

\noindent Step \eqref{item:control-E2-2}: By \eqref{P-Markov}(c) it
is possible to define $V=V(N,k)$ as the total number of sets
$\Upsilon_{j_1,\ldots,j_k}$ such that $R(\Upsilon_{j_l})\leq N$ for
all $i=1,\ldots,k$. Now, using \eqref{u-matching-s-sublattices}, we
may choose $J$ so that for every $n>J$ and
$\Upsilon_{j_1,\ldots,j_k}^0$ such that $R_0(\Upsilon_{j_l}^0)\leq
N$ for all $i=1,\ldots,k$ then the corresponding
$\Upsilon_{j_1,\ldots,j_k}^{n}$ is such that
\[
\l0\left(\Upsilon^0_{j_1,\ldots,j_k}\bigtriangleup
H_{n}\left(\Upsilon^{n}_{j_1,\ldots,j_k}\right)\right)<
\frac{\varepsilon}{3}\, V^{-1}\,(2\max\{1,\|g\|_\infty\})^{-1}.
\]
Under these circumstances we have
\[
\sum_{\begin{tabular}{c}
            {\tiny $j_1,\ldots,j_k$}:\\
            {\tiny $R_0(\Upsilon_{j_l}^0)\leq N$}\\
            {\tiny $l=1,\ldots,k$}
            \end{tabular}}
\int_{\Upsilon^0_{j_1,\ldots,j_k}\bigtriangleup
H_{n}\left(\Upsilon^{n}_{j_1,\ldots,j_k}\right)} \left|(g\circ
f_{n}^\ell \circ F_{n}^k)^\bullet
 \circ H_n^{-1} -(g\circ
f_{0}^\ell  \circ F_{0}^k)^\bullet
\right|\I_{\Omega_0\cap\Omega_0^{n}}d\l0< \frac{\varepsilon}{3}.
\]

\noindent Step \eqref{item:control-E2-3}: For each $i=1,\ldots,k$,
let $\tau_{j_i}=R_0(\Upsilon_{j_i}^0)$.
 In each set
$\Upsilon_{j_1,\ldots,j_k}^0\cap \Upsilon_{j_1,\ldots,j_k}^{n}$ we
have that $F_0^k=f_0^{\tau_1+\ldots+\tau_k}$ and
$F_{n}^k=f_{n}^{\tau_1+\ldots+\tau_k}$. Since $M$ is compact, each
$f_n$ is $C^k$ and $f_n\to f_0$, as $n\to\infty$, in the $C^k$
topology then
\begin{itemize}
\item there exists $\vartheta>0$ such that $|z-\zeta|<\vartheta\Rightarrow
|g(z)-g(\zeta)|<\frac{\varepsilon}{3}\,V^{-1}$;

\item there exists $J_1$ such that for all $n>J_1$ and $z\in M$ we have
$$\max\left\{|f_0(z)-f_{n}(z)|,\ldots,
|f_0^{kN+l}(z)-f_{n}^{kN+l}(z)|\right\}<\tfrac\vartheta2;$$

\item there exists $\eta>0$ such that for all $z,\zeta\in M$ and $f\in\F$
 $$|z-\zeta|
<\eta\;\Rightarrow\;\max\left\{|f(z)-f(\zeta)|,\ldots,
|f^{kN+l}(z)-f^{kN+l}(\zeta)|\right\}<\tfrac\vartheta2.$$

\end{itemize}
Furthermore, according to \eqref{u-proximity-stable-direction},
\begin{itemize}
\item there is $J_2$ such that for every $n>J_2$ and $x
\in\Omega_0\cap\Omega_0^{n}$ we have
\[
\left|
\gamma_0^s(x)-\gamma_{n}^s(x)\right|_{C^1}<\eta.
\]
\end{itemize}

Let $n>\max\{J_1,J_2\}$, $z\in\gamma_0^s(x)$ and
 take $\zeta\in\gamma_{n}^s(x)$ such that $|z-\zeta|<\eta$.
This together with the choices of
$\eta$ and $J_1$ implies
\begin{equation*}
\begin{split}
\left|f_0^\ell\circ F_0^k(z)-f_{n}^\ell\circ F_{n}^k(\zeta)
\right|&\leq
\left|f_0^{\tau_1+\ldots+\tau_k+l}(z)-f_{0}^{\tau_1+\ldots+\tau_k+l}(\zeta)
\right|\\&\quad+
\left|f_0^{\tau_1+\ldots+\tau_k+l}(\zeta)-f_{n}^{\tau_1+\ldots+\tau_k+l}(\zeta)
\right|
\\
&< \vartheta/2+\vartheta/2=\vartheta.
\end{split}
\end{equation*}
Finally, the above considerations and the choice of $\vartheta$
allow us to conclude that for every $n>\max\{J_1,J_2\}$,
$x\in\Omega_0\cap\Omega_0^{n}$ and
$z\in\gamma_0^s(x)$, there exists
$\zeta\in\gamma_{n}^s(x)$ such that
\begin{equation}
\label{eq:control-E2-5} \left|g(f_{n}^\ell\circ
F_{n}^k(\zeta))-g(f_0^\ell\circ F_0^k(z))\right|<
\frac{\varepsilon}{3}\,V^{-1}.
\end{equation}
Attending to \eqref{eq:def-discretization}, \eqref{eq:control-E2-5}
and the fact that we can interchange the roles of $z$ and $\zeta$ in
the latter, we obtain that for every $n>\max\{J_1,J_2\}$
\[
\left|(g\circ f_{n}^\ell \circ F_{n}^k)^\bullet
 \circ H_n^{-1} -(g\circ
f_{0}^\ell  \circ F_{0}^k)^\bullet
\right|<\frac{\varepsilon}{3}\,V^{-1},
\]
from where we deduce that
\[
\sum_{\begin{tabular}{c}
           {\tiny $j_1,\ldots,j_k$}\\
            {\tiny $R_0(\Upsilon_{j_l}^0)\leq N$}\\
            {\tiny $1\le l\le k$}
            \end{tabular}}
\int_{\Upsilon^0_{j_1,\ldots,j_k}\bigtriangleup
H_{n}\left(\Upsilon^{n}_{j_1,\ldots,j_k}\right)} \left|(g\circ
f_{n}^\ell \circ F_{n}^k)^\bullet
 \circ H_n^{-1} -(g\circ
f_{0}^\ell  \circ F_{0}^k)^\bullet
\right|\I_{\Omega_0\cap\Omega_0^{n}}d\l0< \frac{\varepsilon}{3}.
\]
\end{proof}

\begin{lemma}
\label{lem:control-of-E3} Given $l,k\in\N$ and  $\varepsilon>0$
there exists $J\in\N$ such that for every $n>J$
\[
\int\left|(\I_{\{R_{n}>\ell \}}\circ F_{n}^k)^\bullet \circ
H_n^{-1}-(\I_{\{R_{0}>\ell \}}\circ F_{0}^k)^\bullet
\right|\I_{\Omega_0\cap\Omega_0^{n}}d\l0 <\varepsilon.
\]
\end{lemma}
\begin{proof}
As in the proof of Lemma \ref{lem:control-of-E2}, we divide the
argument into three steps.

\eqref{item:control-E2-1} The condition on $N$: Consider the set
\[
L_1=\left\{x\in\Omega_0\cap\Omega_0^{n}:\,\exists
t\in\{1,\ldots,k+1\}\,\mbox{ such that }\, R_0^t(x)>N \,\mbox{ or
}\, R_{n}^t(x)>N\right\}.
\]
From Lemma \ref{lem:choice-N5} we have  $|L_1|\leq  (k+1)
C_1.\left(\l0(\{R_0>N\})+\ln(\{R_{n}>N\})\right)$. So we choose $N$
large enough so that
\[
2\|g\|_\infty (k+1)
C_1.\left(\l0(\{R_0>N\})+\ln(\{R_{n}>N\})\right)<\frac{\varepsilon}{3},
\]
which implies that
\[
\int_{L_1}\left|(\I_{\{R_{n}>\ell \}}\circ F_{n}^k)^\bullet \circ
H_n^{-1}-(\I_{\{R_{0}>\ell \}}\circ F_{0}^k)^\bullet
\right|\I_{\Omega_0\cap\Omega_0^{n}}d\l0< \frac{\varepsilon}{3}.
\]

\eqref{item:control-E2-2} Let as before $V=V(N,k+1)$ be the total
number of sets $\Upsilon_{j_1,\ldots,j_{k+1}}$ such that
$R(\Upsilon_{j_i})\leq N$ for all $i=1,\ldots,k+1$. Now, using
\eqref{u-matching-s-sublattices}, we may choose $J$ so that for
every $n>J$ and $\Upsilon_{j_1,\ldots,j_{k+1}}^0$ such that
$R_0(\Upsilon_{j_i}^0)\leq N$ for all $i=1,\ldots,k+1$ then the
corresponding $\Upsilon_{j_1,\ldots,j_{k+1}}^{n}$ is such that
\[
\l0\left(\Upsilon^0_{j_1,\ldots,j_{k+1}}\bigtriangleup
H_{n}\left(\Upsilon^{n}_{j_1,\ldots,j_{k+1}}\right)\right)<
\frac{\varepsilon}{3}\, V^{-1}\,(2\max\{1,\|g\|_\infty\})^{-1}.
\]
Let
$L_2=\Upsilon^0_{j_1,\ldots,j_{k+1}}\bigtriangleup
H_{n}\left(\Upsilon^{n}_{j_1,\ldots,j_{k+1}}\right)$ and observe that
\[
\sum_{\begin{tabular}{c}
            {\tiny $j_1,\ldots,j_{k+1}$}:\\
            {\tiny $R_0(\Upsilon_{j_l}^0)\leq N$}\\
            {\tiny $l=1,\ldots,k+1$}
            \end{tabular}}
\int_{L_2} \left|(\I_{\{R_{n}>\ell \}}\circ F_{n}^k)^\bullet \circ
H_n^{-1}-(\I_{\{R_{0}>\ell \}}\circ F_{0}^k)^\bullet
\right|\I_{\Omega_0\cap\Omega_0^{n}}d\l0< \frac{\varepsilon}{3}.
\]

\eqref{item:control-E2-3} At last, notice that in each set
$\Upsilon^0_{j_1,\ldots,j_{k+1}}\cap
H_{n}\left(\Upsilon^{n}_{j_1,\ldots,j_{k+1}}\right)$ we have
\[
\left|(\I_{\{R_{n}>l\}}\circ F_{n}^k)^\bullet\circ
H_{n}^{-1}-(\I_{\{R_{0}>l\}}\circ F_{0}^k)^\bullet\right|=0,
\]
which gives the result.
\end{proof}

\section{Entropy continuity}\label{sec:entropy-continuity}

In Proposition~\ref{prop:entropy-formula} we have seen that the SRB entropy can be written just in terms of the quotient dynamics.  Our aim now is to show that the integrals appearing in that formula are close for nearby dynamics, and this is the content of Proposition~\ref{prop:convergence-entropy}.
Notice that since the integrands are not necessarily continuous functions, the continuity of the integrals is not an immediate consequence of the statistical stability.

\subsection{Auxiliary results}\label{subsec:auxiliary-lemmas}

\begin{lemma}\label{lem:l1l1}
Let $(\varphi_n)_{n\in\N}$ be a bounded sequence of $m$-measurable
functions defined on $M$ belonging to $L^\infty(m)$. If
$\varphi_n\to\varphi$ in the $L^1(m)$-norm and $\psi\in L^1(m)$,
then
\[
\int\psi(\varphi_n-\varphi)dm\to0, \quad\mbox{when $n\to\infty$}.
\]
\end{lemma}

\begin{proof}
Take any $\varepsilon>0$. Let $C>0$ be an upper bound for
$\|\varphi_n\|_\infty$. Since $\psi\in L^1(m)$, there is $\delta>0$
such that for any Borel set $B\subset M$
\begin{equation}
\label{eq:lema-l1l1} m(B)<\delta\quad\Rightarrow\quad \int_B
|\psi|dm<\frac\varepsilon{4C}.
\end{equation}
Define for each $n\geq1$
\[
B_n=\left\{x\in M:
|\varphi_n(x)-\varphi_0(x)|>\frac\varepsilon{2\|\psi\|_1}\right\}.
\]
Since $\|\varphi_n-\varphi_0\|_1\to0$ when $n\to\infty$, then there
is $n_0\in\N$ such that $m(B_n)<\delta$ for every $n\geq n_0$.
Taking into account the definition of $B_n$, we may write
\begin{align*}
  \int |\psi||\varphi_n-\varphi_0|dm&=\int_{B_n}|\psi||\varphi_n-\varphi_0|
  dm+\int_{M\setminus B_n} |\psi||\varphi_n-\varphi_0|dm\\
  &\leq2C\int_{B_n}|\psi|dm+\frac
  \varepsilon{2\|\psi\|_1}\int_{M\setminus B_n}|\psi|dm.
\end{align*}
Then, using \eqref{eq:lema-l1l1}, this last sum is upper bounded by
$\varepsilon$, as long as $n\geq n_0$.
\end{proof}

\begin{lemma}\label{lem:aux-lemma-bound-JF} There is $C_2>0$ such that
$\log J\bar F_n\le C_2\,R_n$ for every $n\ge 0$.
\end{lemma}

\begin{proof} Define   $L_n=\max_{x\in M}\{|\det Df_n^u(x)|\}$, for each $n\ge 0$.
  By the compactness  of $M$ and the
continuity on the first order derivative,  there is $L > 1$ such
that $L_n\le L$ for all $n\ge 0$. We have
 $$
 |\det D(\f_n)^u(x)|=\prod_{j=0}^{R_n(x)-1}|\det Df_n^u(f_n^j(x))|
 \le L^{R_n(x)}.
 $$
By \eqref{eq:jacobians-rel} it follows that
\[
\log J(F_n)(x)=\log|\det DF_n^u(x)|+\log \hat u(F_n(x))-\log \hat
u(x).
\]
Observing that by \eqref{P-regularity-s-foliation}(a) it follows
that $|\log \hat u(F_n(x))-\log \hat u(x)|\leq 2C\beta^0=2C$, we have
 $$\log J(F_n)(x) \le R_n(x)\log L+2C.$$
To conclude, we take  $C_2=\log L+2C$.
\end{proof}

\begin{lemma}\label{le.impu1}
Given $\varepsilon>0$, there
is $J\in\N$ such that for all $n>J$
\[
\int_{\Omega^n_0\cap\Omega_0}|R_n-R_0|\,d\l0\leq\varepsilon
\]
 \end{lemma}

\begin{proof}
Let $\varepsilon>0$ be given. Using condition \eqref{u-uniform-tail}
and
Remark~\ref{rem:Leb-Hm*ln-l0}, take $N\geq 1$ and $J=J(N,\varepsilon)>0$ in such a way
that $\sum_{j=N}^\infty j\l0 \{R_n=j\}<\varepsilon/3$ and
$\sum_{j=N}^\infty j\l0 \{R_0=j\}<\varepsilon/3$. Since
$$R_n=\sum_{j=0}^{\infty}\I_{\{R_n>j\}},$$
 we may write
 \begin{eqnarray*}
 \|R_n-R_0\|_1
 &=&
 \big\|R_n
 -\sum_{j=0}^{N-1}\I_{\{R_n>j\}}
 +\sum_{j=0}^{N-1}\big(\I_{\{R_n>j\}}
 -\I_{\{R_0>j\}}\big)
 +\sum_{j=0}^{N-1}\I_{\{R_0>j\}}
 -R_0\big\|_1 \\
 &\leq&
 \big\|
 \sum_{j=N}^{\infty}\I_{\{R_n>j\}}\big\|_1
 +\sum_{j=0}^{N-1}\|\I_{\{R_n>j\}}
    -\I_{\{R_0>j\}}\|_1
 +
 \big\|\sum_{j=N}^{\infty}\I_{\{R_0>j\}}
 \big\|_1\\
 &=&\big\|
 \sum_{j=N}^{\infty}\I_{\{R_n> j\}}\big\|_1
 +\sum_{j=0}^{N-1}\|\I_{\{R_n\leq j\}}
    -\I_{\{R_0\leq j\}}\|_1
 +
 \big\|\sum_{j=N}^{\infty}\I_{\{R_0>j\}}
 \big\|_1.
 \end{eqnarray*}
 By the choices of $N$ and $J$,
 the first and third terms in this last sum are  smaller than
 $\varepsilon/3$.
 By \eqref{u-matching-s-sublattices}, increasing $J$ if necessary, we can make
 $\l0\left(\{R_{n}=j\} \triangle
\{R_{0}=j\}\right)$ sufficiently small in order to have the second
term smaller than $\epsilon/3$.
\end{proof}

\subsection{Convergence of metric
entropies}\label{subsec:conv-entropies}

Our aim is to show that $h_{\mu_{n}}\to h_{\mu_0}$ as $n\to \infty$, which by
Proposition~\ref{prop:entropy-formula}  can be rewritten as
\begin{equation}\label{eq.sigma}
\sigma_n^{-1}\int_{\bar\Lambda_n}\log
J\bar F_n\,d\bar \mu_n\longrightarrow \sigma_0^{-1}\int_{\bar\Lambda_0}\log J\bar F_0\,d\bar
\mu_0,\quad\text{as $n\to\infty$}.
\end{equation}
 Observing that $\sigma_n=\int_{\Lambda_n}R_n d\tilde\mu_n=\mu_n^*(M)$,
then by
Proposition~\ref{prop:convergence-SRB-measures} we have $\sigma_n\to \sigma_0$, as $n\to \infty$. Hence, \eqref{eq.sigma} is a consequence of the next result.

\begin{proposition}\label{prop:convergence-entropy}
$\displaystyle\int_{\bar\Lambda_n}\log
J\bar F_n\,d\bar \mu_n\longrightarrow \int_{\bar\Lambda_0}\log J\bar F_0\,d\bar
\mu_0$ as $n\to\infty$.
\end{proposition}
\begin{proof}
The convergence above will follow if we show that the following term
is arbitrarily small for large $n\in\N$.
\[
E:=\left|\int_{\Omega_n}(\log J\bar F_n\circ\hat\pi_n)(\bar\rho_n\circ\hat\pi_n)\,d\ln-
\int_{\Omega_0}(\log J\bar F_0\circ\hat\pi_0)\varrho_0\,d\l0\right|.
\]
Recall that $\varrho_0=\bar\rho_0\circ\hat\pi_0$ and
$\varrho_n=\bar\rho_n\circ\hat\pi_n\circ H_n^{-1}$, for every
$n\in\N$. Define
\begin{equation*}
E_0:=\left|\int_{\Omega^n_0\cap\Omega_0}(\log J\bar F_n\circ\hat\pi_n\circ H_n^{-1})\varrho_n\,d(H_n)_*\ln-
\int_{\Omega^n_0\cap\Omega_0}(\log J\bar F_0\circ\hat\pi_0)\varrho_0\,d\l0\right|.
\end{equation*}
By Lemmas~\ref{lem:small-adjust} and \ref{lem:aux-lemma-bound-JF} we
have
$$
E\leq E_0+KC_2\int_{\Omega_n\setminus\Omega_n^0}R_n d\ln+KC_2
\int_{\Omega_0\setminus\Omega^n_0}R_0d\l0.
$$
Since $R_0\in L^1(\l0)$, then, by \eqref{u-matching-cantor-sets} and
Remark~\ref{rem:Leb-Hm*ln-l0}, for large $n$, we may have
$\l0(\Omega_0\triangle\Omega^n_0)$ small so that $
\int_{\Omega_0\setminus\Omega^n_0}R_0d\l0$ becomes negligible.
Now, for each $N\in\N$
\[
\int_{\Omega_n\setminus\Omega_n^0}R_n d\ln\leq
N\int_{\Omega_n\setminus\Omega_n^0}
d\ln+\int_{\{R_n>N\}}R_nd\ln.
\]
Using condition \eqref{u-uniform-tail} we may choose $N$ so that for
all $n\in\N$ large enough the quantity
$\int_{\{R_n>N\}}R_nd\ln=\sum_{j=N+1}j\l0\{R_n=j\}$ is arbitrarily
small. Again, using \eqref{u-matching-cantor-sets}, if $n\in\N$ is
sufficiently large then $\int_{\Omega^n_0\setminus\Omega_0}d\l0$
is as small as we want. Therefore, we are reduced to estimating
$E_0$.

Note that by definition $\Omega^n_0\subset\Omega_0$. Having this
in mind, we split $E_0$ into the next three terms that we call
$E_1,E_2,E_3$ respectively.
\begin{align*}
E_0&\leq \left|\int_{\Omega^n_0}
(\log J\bar F_n\circ\hat\pi_n\circ H_n^{-1})\varrho_n\,d(H_n)_*\ln-
\int_{\Omega^n_0}
(\log J\bar F_0\circ\hat\pi_0)\varrho_n\,d(H_n)_*\ln\right|\\
&\quad+\left|\int_{\Omega^n_0}
(\log J\bar F_0\circ\hat\pi_0)\varrho_n\,d(H_n)_*\ln-\int_{\Omega^n_0}
(\log J\bar F_0\circ\hat\pi_0)\varrho_n\,d\l0 \right|\\
&\quad +\left|\int_{\Omega^n_0}
(\log J\bar F_0\circ\hat\pi_0)\varrho_n\,d\l0-\int_{\Omega^n_0}
(\log J\bar F_0\circ\hat\pi_0)\varrho_0\,d\l0\right|.
\end{align*}
Concerning $E_2$, using Lemma~\ref{lem:small-adjust} and
Lemma~\ref{lem:aux-lemma-bound-JF} we have
\begin{align*}
E_2&\leq\int_{\Omega^n_0}|\log J\bar F_0||\varrho_n|\left|\frac{d(H_n)_*\ln}{d\l0}-1\right|d\l0\\
& \leq K C_2 \int_{\Omega^n_0}
R_0\left|\frac{d(H_n)_*\ln}{d\l0}-1\right|d\l0.
\end{align*}
Now, Remark~\ref{rem:Leb-Hm*ln-l0} and Lemma~\ref{lem:l1l1}
guarantee that $E_2$ can be made arbitrarily small for
 large $n\in\N$.
Using Corollary~\ref{co:conv-densities-weak*},
$E_3$ can also be made small for large $n$.
We are left with $E_1$. By Lemma~\ref{lem:small-adjust} and Remark~\ref{rem:Leb-Hm*ln-l0}
we only need to control
$$
\int_{\Omega^n_0\cap\Omega_0}\left|(\log J\bar F_n\circ\hat\pi_n\circ H_n^{-1})-(\log J\bar F_0\circ\hat\pi_0)\right|\,d\l0
$$ whose estimation
we leave to Lemma~\ref{lem:control-E1-entropy}.
\end{proof}

\begin{remark} \label{rem:chato}
Assume that $\gamma_n$ is a
 compact unstable  manifold of the map $f_n$  for $n\ge 0$ and
$\gamma_n\to\gamma_0$, in the $C^1$ topology. The convergence of $f_n$ to $f_0$ in the $C^1$ topology ensures that given $\ell\in\N$ and
$\epsilon>0$ there exist $\delta=\delta(\ell,\epsilon)>0$ and
$J=J(\delta)\in\N$ such that for every $n>J$, $x\in\gamma_0$ and
$y\in\gamma_n$ with $|x-y|<\delta$
\begin{equation*}\label{eq:estimate-2}
\max_{j=1,\ldots,\ell}\bigg\{|f^{j}_n(y)-f^{j}_0(x)|,\,\, |\log\det(Df^{j}_n)^u(y)
-\log\det(Df^{j}_0)^u(x)|\bigg\} <\epsilon.
\end{equation*}
\end{remark}

\begin{lemma}\label{lem:control-E1-entropy}
Given any $\varepsilon>0$ there exists $J\in\N$ such that for every
$n>J$
\[
\int_{\Omega^n_0\cap\Omega_0}\left|(\log J\bar F_n\circ\hat\pi_n\circ H_n^{-1})-(\log J\bar F_0\circ\hat\pi_0)\right|\,d\l0<\varepsilon.
\]
\end{lemma}
\begin{proof}
Let $\varepsilon>0$ be given. For $n,N\in\N$ define
$A_{n,N}=\{R_n\leq N\}\cap\{R_0\leq N\}$ and
$A_{n,N}^c=\{R_n> N\}\cup\{R_0> N\}$. By
Lemma~\ref{lem:aux-lemma-bound-JF} we have
\begin{multline*}
\int_{\Omega^n_0\cap A^c_{n,N}}\left|(\log J\bar F_n\circ\hat\pi_n\circ H_n^{-1})-(\log J\bar F_0\circ\hat\pi_0)
\right|\,d\l0\leq
C_2\int_{\Omega^n_0\cap A^c_{n,N}}R_n\,d\l0\\+
C_2\int_{\Omega^n_0\cap A^c_{n,N}}R_0\,d\l0.
\end{multline*}
Since $R_0\in L^1(\l0)$, there is $\delta>0$ such that if a
measurable set $A$ has $\l0(A)<\delta$, then $\int_A
R_0d\l0<\varepsilon/(4C_2)$. According to \eqref{u-uniform-tail}, we
may pick $N\in\N$ and choose $J\in\N$ such that for every $n>J$ we
get $\l0( A^c_{n,N})<\delta$. This implies that the second term on
the right hand side of the inequality above is smaller than
$\varepsilon/4$. The same argument and Lemma~\ref{le.impu1} allow us
to conclude that for a convenient choice of $N\in\N$ and for
$J\in\N$ sufficiently large
\[
C_2\int_{\Omega^n_0\cap A^c_{n,N}}R_n\,d\l0\leq
C_2\int_{\Omega^n_0\cap A^c_{n,N}}R_0\,d\l0+
C_2\int_{\Omega^n_0}|R_n-R_0|\,d\l0\leq\frac\varepsilon4.
\]
So, assuming that $N$ has been chosen and $J$ is sufficiently large so
that
\[
\int_{\Omega^n_0\cap A^c_{n,N}}\left|(\log J\bar F_n\circ\hat\pi_n\circ H_n^{-1})-(\log J\bar F_0\circ\hat\pi_0)\right|\,d\l0\leq \varepsilon/2,
\]
we are left do deal with
\begin{multline*}
\int_{\Omega^n_0\cap  A_{n,N}}\left|(\log J\bar F_n\circ\hat\pi_n\circ H_n^{-1})-(\log J\bar F_0\circ\hat\pi_0)\right|\,d\l0\leq\\
\sum_{i:R_0(\Upsilon_i^0)\leq
N}\int_{\Upsilon_i^0\cap\Upsilon_i^n}\left|(\log J\bar F_n\circ\hat\pi_n\circ H_n^{-1})-(\log J\bar F_0\circ\hat\pi_0)\right|\I_{\Omega^n_0\cap\Omega_0}\,d\l0\\
+\sum_{i:R_0(\Upsilon_i^0)\leq
N}\int_{\Upsilon_i^0\triangle\Upsilon_i^n}\left|(\log J\bar F_n\circ\hat\pi_n\circ H_n^{-1})-(\log J\bar F_0\circ\hat\pi_0)
\right|\I_{\Omega^n_0\cap\Omega_0\cap  A_{n,N}}\,d\l0.
\end{multline*}
Denote by $S_1$ and $S_2$ respectively the first and second sums
above, and $v$  the number of terms in $S_1$ and $S_2$.
By Lemma~\ref{lem:aux-lemma-bound-JF} we have
 $$
 S_2\le
C_2\int_{\Upsilon_i^0\triangle\Upsilon_i^n}(R_n+R_0)
\I_{\Omega^n_0\cap\Omega_0\cap  A_{n,N}}\,d\l0\leq
2C_2N\l0(\Upsilon_i^0\triangle\Upsilon_i^n).
 $$
Hence, using \eqref{u-matching-s-sublattices} we consider $J\in\N$
large enough to have
$\l0(\Upsilon_i^0\triangle\Upsilon_i^n)<\varepsilon/(8C_2Nv)$, and
so $S_2\leq \varepsilon/4$.

Let $\tau_i=R_0(\Upsilon_i^0)=R_n(\Upsilon_i^n)\leq N$. We want to see that for all $n$ large enough and all
$x\in\Upsilon_i^0\cap\Upsilon_i^n$ with $\tau_i\leq N$
\begin{equation}\label{eq:final-goal}
\left|(\log J\bar F_n\circ\hat\pi_n\circ H_n^{-1})(x)-(\log J\bar F_0\circ\hat\pi_0)(x)\right|\leq \varepsilon/4v,
\end{equation}
which yields $S_1\leq\varepsilon/4$.
 Using \eqref{eq:jacobians-rel} and observing
that the curves $\hat\gamma_n,\hat\gamma_0$ are the leaves we chose
to define the reference measures $\bar m_n, \bar m_0$, then we
easily get for $y=H^{-1}_n(x)$
\begin{align*}
\left|\log J\bar F_n\circ\hat\pi_n (y)-\log J\bar F_0\circ\hat\pi_0(x)\right|&\leq
\left|\log\det(Df^{\tau_i}_n)^u(y)
-\log\det(Df^{\tau_i}_0)^u(x)\right|\\
&\quad+|\log\hat u_n(f_n^{\tau_i}(y))-\log\hat
u_0(f_0^{\tau_i}(x))|.
\end{align*}
Using Remark~\ref{rem:chato} with $\ell=N$ and
$\varepsilon/8v$ instead of $\epsilon$, and recalling that $\tau_i\leq N$, we may find
$\delta>0$ and $J\in\N$ so that for all $n>J$
\begin{equation}\label{eq:estimate-3}
\left|\log\det(Df^{\tau_i}_n)^u(y)
-\log\det(Df^{\tau_i}_0)^u(x)\right|<\varepsilon/8v.
\end{equation}
Observe that $|x-y|<\delta$ as long as $J$ is sufficiently large, since $x=H_n(y)$.

For every $n,k\in\N_0$ and $t\in\Lambda_n$, let $$\hat
u_n^{k}(t)=\prod_{j=0}^{k}\frac{\det Df_n^u(f_n^j(t))}{\det
Df_n^u(f_n^j(\hat t))}.$$ By definition of $\hat u_n$ (see
\eqref{eq:def-u-hat}) and by \eqref{P-regularity-s-foliation}(a),
 there is $k\in\N$ such that for every
$n\in\N_0$ and $t\in\Lambda_n$ we have $|\log\hat u_n(t)-\log\hat
u_n^{k}(t)|<\varepsilon/(48v)$. Thus,
\begin{eqnarray*}
|\log\hat u_n(f_n^{\tau_i}(y))-\log\hat
u_0(f_0^{\tau_i}(x))|&\leq &|\log\hat u_n(f_n^{\tau_i}(y))-\log\hat
u_n^{k}(f_n^{\tau_i}(y))|\\
 &\quad&+|\log\hat u_n^{k}(f_n^{\tau_i}(y))-\log\hat
u_0^{k}(f_0^{\tau_i}(x))|\\
 &\quad&+|\log\hat u_0^{k}(f_0^{\tau_i}(x))-\log\hat
u_0(f_0^{\tau_i}(x))|\\
&\leq&\sum_{j=0}^{k}\left| \log\det
Df_n^u(f_n^j(\zeta))-\log\det Df_0^u(f_0^j(z))\right|\\
&\quad&+\sum_{j=0}^{k}\left| \log\det Df_n^u(f_n^j(\hat\zeta))-\log\det
Df_0^u(f_0^j(\hat z))\right|\\
&\quad&+\frac\varepsilon{24v},
\end{eqnarray*}
where $z=f_0^{\tau_i}(x)$, $\zeta=f_n^{\tau_i}(y)$, $\hat z$ is the
only point on the set $\gamma_0^s(z)\cap\hat\gamma_0$ and
$\hat\zeta$ is the unique point on the set
$\gamma_n^s(\zeta)\cap\hat\gamma_n$.

Observe that since $\hat\gamma_n\to\hat\gamma_0$ and $f_n\to f_0$ in
the $C^1$ topology, and $\tau_i\leq N$, then
$\gamma_n^u(\zeta)\to\gamma_0^u(z)$, in the $C^1$ topology. 
Besides, using
Lemma~\ref{lem:bar-F0-bar-Fni-proximity}
 we also have $|\hat z-\hat\zeta|$ as small as we want for $J$ large enough. Consequently, by Remark~\ref{rem:chato}, we may
find $J\in\N$
sufficiently large so that for all $n>J$, we have
\begin{equation}\label{eq:estimate-4}
\sum_{j=0}^{k}\left| \log\det Df_n^u(f_n^j(\zeta))-\log\det
Df_0^u(f_0^j(z))\right|<\varepsilon/(24v).
\end{equation}
and
\begin{equation}\label{eq:estimate-5}
\sum_{j=0}^{k}\left| \log\det Df_n^u(f_n^j(\hat\zeta))-\log\det
Df_0^u(f_0^j(\hat z))\right|<\varepsilon/(24v).
\end{equation}
Estimates
\eqref{eq:estimate-3},\eqref{eq:estimate-4} and
\eqref{eq:estimate-5} yield \eqref{eq:final-goal}.
\end{proof}

\bibliographystyle{alpha}

\end{document}